\documentclass[12pt,a4]{article}

\usepackage{amsmath}
\usepackage{epsfig,amsthm}
\usepackage{amssymb,amsfonts}
\usepackage{dsfont}
\usepackage{graphicx}
\usepackage{subfigure}
\usepackage{indentfirst}
\usepackage{mathrsfs}
\usepackage{url}
\usepackage{longtable}
\usepackage{xparse}
\usepackage{bm}
\usepackage{bbm}
\usepackage{siunitx}
\usepackage{cite}
\usepackage{authblk}
\usepackage{color}
\usepackage{cleveref}

\NewDocumentCommand{\tensor}{t_}
 {%
  \IfBooleanTF{#1}
    {\tensop}
    {\otimes}%
 }
\NewDocumentCommand{\tensop}{m}
 {%
  \mathbin{\mathop{\otimes}\displaylimits_{#1}}%
 }
 
\providecommand{\keywords}[1]{\textbf{\textit{Keywords---}} #1}

\newcommand{\D}{\mathrm{d}}

\newcommand{\E}{\mathrm{e}}
\newcommand{\eps}{\varepsilon}

\newcommand{\T}{\intercal}

\newcommand*\diff{\mathop{}\!\D}

\newcommand{\abs}[1]{\lvert#1\rvert}

\newcommand{\norm}[1]{\lVert#1\rVert}

\newcommand{\vxi}{\bm{\xi}}

\newcommand{\vk}{\bm{k}}

\newcommand{\vp}{\bm{p}}
\newcommand{\vq}{\bm{q}}

\newcommand{\vu}{\bm{u}}
\newcommand{\vv}{\bm{v}}

\newcommand{\vx}{\bm{x}}
\newcommand{\vy}{\bm{y}}

\newcommand{\vB}{\bm{B}}

\newcommand{\sR}{\mathbb{R}}

\newcommand{\sZ}{\mathbb{Z}}

\newtheorem{thm}{Theorem}
\newtheorem*{thmm}{Theorem}

\newtheorem*{defi}{Definition}

\newtheorem{prop}{Proposition}
\newtheorem{rmk}{Remark}

\title{Existence, uniqueness, and energy scaling of 2+1 dimensional  continuum model for stepped epitaxial surfaces with elastic effects}

\author{
Ganghua Fan$^{1}$\thanks{E-mail: gfanab@connect.ust.hk},
Tao Luo$^{2}$\thanks{E-mail: luotao41@sjtu.edu.cn},
Yang Xiang$^{1}$\thanks{E-mail: maxiang@ust.hk}
\\
$^{1}$ Department of Mathematics, Hong Kong University of Science and Technology, Clear Water Bay, Hong Kong\\
$^{2}$ School of Mathematical Sciences, Institute of Natural Sciences, MOE-LSC, and Qing Yuan Research Institute, Shanghai Jiao Tong University, Shanghai, 200240, P.R. China
}
\date{\today}

\begin{document}
\maketitle
\allowdisplaybreaks

\begin{abstract}
We study the 2+1 dimensional continuum model for  the evolution of stepped epitaxial surface under long-range elastic interaction proposed by Xu and Xiang (SIAM J. Appl. Math. 69, 1393–1414, 2009). The long-range interaction term and the two length scales in this model makes PDE analysis challenging. Moreover, unlike in the $1+1$ dimensional case, there is a nonconvexity contribution in the total energy in the $2+1$ dimensional case, and it is not easy to prove that the solution is always in the well-posed regime during the evolution. In this paper, we propose a modified  2+1 dimensional continuum model based on the underlying physics.
This modification fixes the problem of possible illposedness due to the nonconvexity  of the energy functional.
 We prove the existence and uniqueness of both the static and dynamic solutions and  derive a minimum energy scaling law for them. We show that the minimum energy surface profile is mainly attained by surfaces with step meandering instability. This is essentially different from the energy scaling law for the 1+1 dimensional epitaxial surfaces under elastic effects attained by step bunching surface profiles. We also discuss the transition from the step bunching instability  to the step meandering instability in 2+1 dimensions.
\end{abstract}


\maketitle

\section{Introduction}
\label{sec1}
In epitaxial film growth, elasticity-driven surface morphology instabilities have been widely employed to generate self-assembled nanostructures on the film surfaces, which exhibit interesting electronic and optical properties and have various applications in semiconductor industry~\cite{Misbah2010,Politi2000instabilities}. In heterogeneous epitaxial film, the film has a different lattice constant than that of the substrate, and  the misfit strain causes  step bunching and step meandering instabilities on such a surface. It is important to understand these instability phenomena due to elastic effects for the design and fabrication of advanced materials based on the self-assembly techniques.

In practice, most semiconductor devices are fabricated on vicinal surfaces when the temperature for epitaxial growth is below the roughening transition. In this case, these surfaces consist of a succession of terraces and atomic height steps. Traditional continuum models~\cite{Asaro1972Interface,Grinfeld1986Instability,Srolovitz1989stability} that treated the surface as a continuum cannot be applied directly.
Tersoff et al. \cite{Tersoff1995Step} proposed a discrete model that describes the dynamics of each step. In their model, the elastic interactions between steps include the force dipole caused by the steps and the force monopole caused by misfit stress.  The force dipole stabilizes a uniform step train while the force monopole destabilizes it, leading to the step bunching instability. Duport et al. \cite{Duport1995New} also proposed a discrete model to account for these effects. Besides the dipole and monopole interactions, their model includes the elastic interactions between the adatoms and steps as well as the Schweobel barrier.
In 2+1 dimensions, the elastic effects also lead to step meandering instability  that competes with the bunching instability for straight steps, and these instabilities and their competitions have been examined by Tersoff and  Pehlke~\cite{Tersoff1992},  Houchmandzadeh and Misbah~\cite{Houchmandzadeh1995}, and Leonard and Tersoff~\cite{Leonard2003} using discrete models.

Xiang \cite{Xiang2002Derivation} derived a 1+1 dimensional continuum model for the stepped surfaces with elastic effects by taking the continuum limit from the discrete models~\cite{Duport1995New,Tersoff1995Step}. Instability analysis and  numerical simulations based on this continuum model performed by Xiang and E~\cite{Xiang2004Misfit} showed that this continuum model is able to correctly describe the step bunching instabilities compared with the results of discrete models and experimental observations. Xu and Xiang~\cite{Xu2009Derivation}, Zhu, Xu and Xiang~\cite{Zhu2009Continuum} further developed a 2+1 dimensional continuum model for the stepped surfaces with elastic effects, which is able to account for both the step bunching and step meandering instabilities as well as their competition. Kukta and Bhattacharya \cite{Kukta1999three} proposed a three-dimensional model for step flow mediated crystal growth under stress and terrace diffusion. There are also continuum models for the surfaces in homoepitaxy, which contain only the  force dipole elastic effect, e.g., \cite{AlHajjShehadeh2011evolution,Israeli1999,Liu2019asymmetry,Margetis2006Continuum}.

In Ref.~\cite{Luo2016Energy}, Luo et al. analyzed the step bunching phenomenon in epitaxial growth with elasticity based on the Tersoff's discrete model \cite{Tersoff1995Step}. In this work, a minimum energy scaling law for straight steps was derived and  the one bunch structure was identified. They further extended the analyses to one-dimensional discrete system with general Lennard-Jones type potential \cite{Luo2020energy} as well as one-dimensional continuum model with general Lennard-Jones type potential \cite{Luo2017Bunching}. Dal Maso et al. \cite{DalMaso2014Analytical} and Fonseca et al. \cite{Fonseca2015Regularity} proved the existence and regularity of weak solution of Xiang's continuum model \cite{Xiang2002Derivation}; However, they modified the original PDE, which includes two length scales of $O(1)$ for the overall surface profile and $O(a)$ (with $a\ll 1$ being the lattice constant) for the structure of a step bunch,  to be one of the same length scale for all contributing terms. Gao et al. \cite{Gao2017continuum} proved the first order convergence rate of a modified discrete model  to the strong solution of the limiting PDE, in which all the contributing terms are also on the same length scale. Lu \cite{Lu2018solutions} derived the existence and regularity of strong solution to the evolution equation of Xu and Xiang \cite{Xu2009Derivation} in the radial symmetry case.
All these works lead to better understandings of the elastically-driven self-organized mechanisms. However, analyses of the evolution equation of the epitaxial surfaces in 2+1 dimensions under elastic effects, such as existence, uniqueness and energy scaling laws, are still lacking.

In this paper, we prove the existence and uniqueness of both the static and dynamic solutions and  derive a minimum energy scaling law for the 2+1 dimensional continuum model proposed in \cite{Xu2009Derivation}. The long-range interaction term and the two length scales in this model present challenges for the analysis.  The nonlocal term, characterizing the long-range interaction, is noticed to be related to the $H^{1/2}$ norm.
Moreover, unlike in the $1+1$ dimensional case, there is a nonconvexity contribution (of the gradient norm of the surface height) in the local energy  in the $2+1$ dimensional case (c.f. \cref{subsec:regularization}), and although the ill-posedness associated with such nonconvexity in the continuum model is in general not in the physical regime, it is not easy to prove that the solution is always in the well-posed regime during the evolution. We propose a modified continuum model to fix the inaccurate formulation of the  energy in the small $|\nabla h|$ regime based on the underlying physics. This modification solves the problem of
possible illposedness due to the nonconvexity of the energy functional. The illposedness associated with nonconvexity in the original continuum model is in general not in the physical regime and our modification only leads to negligible change under the physically meaningful setting.
With this modification, we are able to show the convexity of the local energy. This convexity with the Fourier analysis of the nonlocal term allow us to use the direct method in the calculus of variations to show the existence and uniqueness of the energy minimizer.
With further estimation, we also prove the weak solution existence and uniqueness for the evolution equation.

We also obtain a minimum energy scaling law for the 2+1 dimensional epitaxial surfaces under elastic effects. It turns out that the minimum energy surface profile is attained by surfaces with step meandering instability.  This is essentially different from the energy scaling law for the 1+1 dimensional epitaxial surfaces under elastic effects~\cite{Luo2017Bunching}, which is attained by step bunching surface profiles. We also discuss the transition from the step bunching instability  to the step meandering instability in 2+1 dimensions.

The rest of this paper is organized as follows. In \cref{sec2}, we propose a modified continuum model to fix the inaccurate formulation of the  energy in the small $|\nabla h|$ regime based on the underlying physics,
  and introduce a new form of the evolution equation  used in the proofs based on the 2+1 dimensional continuum model derived in \cite{Xu2009Derivation}. In \cref{sec:main}, we state the main analysis results. The existence and uniqueness of the energy minimizer and the weak solution of the evolution equation are shown in \cref{sec3} and \cref{sec4}, respectively. In \cref{sec5}, we prove a minimum energy scaling law in the 2+1 dimensions, and discuss the competition between the step meandering and step bunching instabilities. Conclusions are given in \cref{sec6}. Comparisons in linear instability analysis and numerical simulation by using continuum models with and without the modification are given in Appendix~\ref{appendix:regularization}.

\section{The modified continuum model}
\label{sec2}
In this section, we first briefly review the 2+1 dimensional continuum model for the evolution of stepped epitaxial surfaces under elastic effects obtained in Ref.~\cite{Xu2009Derivation}. We then present a modified form of the energy and accordingly the evolution equation to fix the inaccurate formulation of the  energy in the small $|\nabla h|$ regime based on the underlying physics. This modification eliminates the
possible illposedness due to the nonconvexity (in terms of the gradient of the surface) of the energy functional.
Finally, in order to employ the gradient flow framework in the proofs,
we also introduce an equivalent evolution equation using a new variable instead of original one using  surface height.

\subsection{The original continuum model}
We introduce the original continuum model in $2+1$ dimensions for the evolution of stepped epitaxial surfaces under elastic effects obtained in Ref.~\cite{Xu2009Derivation}. Let $h(\vx)$, $\vx=(x_1,x_2)\in\sR^2$, be the height of the epitaxial surface.
The total energy $E[h]$ consists of three parts: the step line energy $E_\mathrm{l}[h]$, the elastic energy due to force dipole $E_\mathrm{d}[h]$, and the misfit elastic energy $E_\mathrm{m}[h]$. That is,
\begin{equation}
E[h]=E_\mathrm{l}[h]+E_\mathrm{d}[h]+E_\mathrm{m}[h].
\end{equation}
Here
\begin{flalign}
    E_\mathrm{l}[h]
    &= \int_{\sR^2} g_1\abs{\nabla h}\diff{\vx},\\
    E_\mathrm{d}[h]
    &= \int_{\sR^2}\frac{g_3}{3}\abs{\nabla h}^3\diff{\vx},
\end{flalign}
where $g_1$ is the step line energy density, $g_3$ is the strength of the force dipole interaction.

In  heteroepitaxial growth, the lattice misfit  $\epsilon_0$ due to the different lattice constants of the film and the substrate generates a constant misfit stress $\sigma_{11}=\sigma_{22}=\sigma_0=\frac{2G(1+\nu)\epsilon_0}{1-\nu}$ in the film, resulting in the misfit energy
\begin{flalign}\label{eqn:misfitenergy}
E_\mathrm{m}[h]
  =&-\frac{(1-\nu)\sigma_0^2}{4\pi G}\int_{\sR^2} h(\vx)\left[\int_{\sR^2}\frac{\vx-\vy}{\abs{\vx-\vy}^3}\cdot\nabla h(\vy)\diff{\vy}\right]\diff{\vx} \nonumber \\
  &+\frac{(1-\nu)\sigma_0^2 a}{2\pi G}\int_{\sR^2}\abs{\nabla h}\log\frac{2\pi r_c\abs{\nabla h}}{\E a}\diff{\vx},
\end{flalign}
where $G$ is the shear modulus, $\nu$ is the Poisson ratio, $r_c$ is a parameter of order of the size of the core of the step. The first term in $E_\mathrm{m}[h]$ is the traditional expression of the misfit elastic energy above the roughening transition temperature~\cite{Asaro1972Interface,Grinfeld1986Instability,Pimpinelli1998Physics,Srolovitz1989stability}. The second term in $E_\mathrm{m}[h]$ is the contribution to the step line energy incorporating the atomic feature of the stepped surfaces~\cite{Xiang2002Derivation,Xiang2004Misfit,Xu2009Derivation,Zhu2009Continuum}.

Evolution of the epitaxial surface satisfies
\begin{align}\label{eqn:evolution}
  h_t=&\nabla\cdot\left(D\nabla\mu\right),\\
  \mu=&\frac{\delta E[h]}{\delta h},
\end{align}
where $D$ is the mobility constant, $\mu$ is the chemical potential associated with the total energy. Without loss of generality, we set the $D=1$ in this paper.
By direct calculation, the chemical potential $\mu$ is
\begin{flalign}
  \mu[h]=&\frac{\delta E_\mathrm{l}[h]}{\delta h}+\frac{\delta E_\mathrm{d}[h]}{\delta h}+\frac{\delta E_\mathrm{m}[h]}{\delta h} \nonumber \\
  =& -g_1\nabla\cdot\left(\frac{\nabla h}{\abs{\nabla h}}\right)-g_3\nabla\cdot\left(\abs{\nabla h}\nabla h\right)-\frac{(1-\nu)\sigma_0^2}{2\pi G}\int_{\sR^2}\frac{\vx-\vy}{\abs{\vx-\vy}^3}\cdot\nabla h(\vy)\diff{\vy} \nonumber \\
  &-\frac{(1-\nu)\sigma_0^2 a}{2\pi G}\left[\nabla\cdot\left(\frac{\nabla h}{\abs{\nabla h}}\right)\log\frac{2\pi r_c\abs{\nabla h}}{a}+\frac{(\nabla h)^\T (\nabla\nabla h) \nabla h}{\abs{\nabla h}^3}\right], \label{eqn:mu_m}
\end{flalign}
where $(\nabla h)^\T (\nabla\nabla h) \nabla h=h^2_{x_1}h_{x_1x_1}+2h_{x_1} h_{x_2} h_{x_1x_2}+h^2_{x_2} h_{x_2x_2}$ in \cref{eqn:mu_m}.

\begin{rmk}\label{rmk..refernce}
We have another formulation of the misfit energy:
\begin{align}\label{eq..misfitenergytildeh}
     E_\mathrm{m}[h]
  =& -\frac{(1-\nu)\sigma_0^2}{4\pi G}\int_{\sR^2} \tilde{h}(\vx)\left[\int_{\sR^2}\frac{\vx-\vy}{\abs{\vx-\vy}^3}\cdot\nabla \tilde{h}(\vy)\diff{\vy}\right]\diff{\vx} \nonumber\\
  &+\frac{(1-\nu)\sigma_0^2 a}{2\pi G}\int_{\sR^2}\abs{\nabla h}\log\frac{2\pi r_c\abs{\nabla h}}{\E a}\diff{\vx},
\end{align}
where
\begin{equation}
\tilde{h}(\vx):=h(\vx)-\vB^\T\vx
\end{equation}
 is the deviation to the reference plane $\vB^\T\vx$ with $\vB$ being the average gradient of height function. The contribution to the chemical potential $\frac{\delta E_\mathrm{m}[h]}{\delta h}$ remains the same.
\end{rmk}

In this paper, we focus on the periodic setting, and consider the total energy on the periodic cell $\Omega=[0, L]^2$. For non-negative integer $k$, we denote $W^{k,p}_\#(\Omega)$ the Sobolev space of functions whose distributional derivatives up to order $k$ are $\Omega$-periodic and in the space $L^p(\Omega)$. In particular, we also write $H^{k}_\#(\Omega)=W^{k,2}_\#(\Omega)$.

Define the Hilbert space
\begin{equation}
    V:=\left\{\tilde{h}\in H^1_\#(\Omega)\mid \int_{\Omega}\tilde{h}(\vx)\diff{\vx}=0\right\}.
\end{equation}
The solution space is defined as
\begin{equation}
    X:= \left\{h\in H^1(\sR^2)\mid\tilde{h}(\vx):=h(\vx)-\vB^\T\vx\in V\right\}.
\end{equation}
Using the semi-norm on $\Omega$
\begin{equation}
    [\tilde{h}]_{H^{1/2}(\Omega)}
    := \left(\sum\limits_{\vk\in\sZ^2}\abs{\vk}\abs{h_{\vk}}^2\right)^{1/2}
\end{equation}
with $h_{\vk}$ being the Fourier coefficient of $\tilde{h}$, the double-integral term in $E_\mathrm{m}[h]$ in Eq.~\eqref{eq..misfitenergytildeh} (or Eq.~\eqref{eqn:misfitenergy}) over one periodic domain $\Omega$ can be written in terms of the semi-norm $H^{\frac{1}{2}}$ as
\begin{equation}\label{nonlocal-energy-1}
\frac{(1-\nu)\sigma_0^2}{4\pi G}\int_{\Omega} h(\vx)\left[\int_{\sR^2}\frac{\vx-\vy}{\abs{\vx-\vy}^3}\cdot\nabla h(\vy)\diff{\vy}\right]\diff{\vx}=\frac{(1-\nu)\sigma_0^2\pi}{ G} L[\tilde{h}]_{H^{1/2}(\Omega)}^2.
\end{equation}

Let $c_1=\frac{(1-\nu)\sigma_0^2}{2\pi G}$, $ac_2=g_1+ac_1\log\frac{2\pi r_c}{\E a}$ and $c_3 a=\frac{g_3}{3}$. Note that $c_1, c_3>0$. The total energy thus can be expressed as
$$  E[h]
   = -2c_1\pi^2L[\tilde{h}]_{H^{1/2}(\Omega)}^2+a \int_{\Omega}(c_1|\nabla h|\log|\nabla h|+c_2|\nabla h|+c_3|\nabla h|^3)\diff{\vx}.$$
We write it as
\begin{flalign}\label{eq..EnergyPeriodic}
  E[h]
   = -2c_1\pi^2L[\tilde{h}]_{H^{1/2}(\Omega)}^2+\int_{\Omega}\Psi_0(\nabla h)\diff{\vx},
   \end{flalign}
where
\begin{flalign}\label{eq..energylocal}
  \Psi_0(\vp) =    ac_1\abs{\vp}\log{\abs{\vp}}+ac_2\abs{\vp}+ac_3\abs{\vp}^3.
\end{flalign}
Here $\Psi_0(\nabla h)$ is the local energy density.

\subsection{Modified continuum model}\label{subsec:regularization}

Consider the local energy density $\Psi_0(\nabla h)$ in Eqs.~\eqref{eq..EnergyPeriodic} and \eqref{eq..energylocal}. As shown in \cref{prop..Convexity} and illustrated in \cref{fig..modifytoconvex}(a), $\Psi_0(\nabla h)$ is not convex when $\nabla h$ is small.
If we consider $\Psi_0(\nabla h)$ as the density of a generalized step line energy:
\begin{equation}\label{eqn:g-line energy}
\Psi_0(\nabla h)=a(c_1\log|\nabla h|+c_2+c_3|\nabla h|^2)|\nabla h|,
\end{equation}
it is negative when $\nabla h$ is small. Here $|\nabla h|$ is the step density and the prefactor of it in the above formula is considered as a generalized step line energy.  As a result, the evolution equation is not well-posed  when $\nabla h$ is small. In fact,
when $\nabla h$ is very small, the leading order contribution in the chemical potential $\mu$ in Eq.~\eqref{eqn:mu_m} is the $\log |\nabla h|$ term, and the leading order of the evolution equation \eqref{eqn:evolution} is
\begin{equation}
h_t=\frac{D(1-\nu)\sigma_0^2 a}{2\pi G}\Delta\left(\nabla\cdot\left(\frac{\nabla h}{\abs{\nabla h}}\right)\log\frac{1}{\abs{\nabla h}}\right),
\end{equation}
which is ill-posed  (of $h_t=\Delta^2h$ type).

Recall that the $2+1$ dimensional continuum model was derived from discrete step model by asymptotic analysis (coarse graining) for vicinal surfaces that consists of a series of monotonic steps~\cite{Xu2009Derivation}. There is no such negative step line energy and the associated instability (corresponding to the illposedness in the continuum model) in the discrete model.
In fact, almost all such continuum models for stepped surfaces (e.g. those reviewed in the introduction section) are based on the approximation that inter-step distance in the continuum model is given by $a/|\nabla h|$.
Physically, such distance cannot exceed domain size $L$, i.e., we should have $a/|\nabla h|\leq L$ or $|\nabla h|\geq a/L$. However, once the continuum is established, there is no guarantee that $|\nabla h|$ is always bounded below by this positive constant during the evolution.
On the other hand,
when the inter-step distance in the discrete model is large, although the continuum model still provides a way to connect the neighboring steps smoothly,  the smooth surface profile $h$ may vary due to the continuum equation over the large terrace region between such neighboring steps, the approximation of inter-step distance by $a/|\nabla h|$ may fail over such a region, i.e., $a/|\nabla h|$ no longer means inter-step distance in this case. It should be fine if such a smooth connection in the continuum model does not introduce irrelevant physical effects. Unfortunately, this is not the case here because very small $\nabla h$ leads to illposedness of the continuum model as demonstrated above.

A straightforward idea to fix this problem is to show that $|\nabla h|$ is always bounded below by some positive constant during the evolution by the continuum model. However, such analysis is challenging here, especially due to the nonlocal term in the evolution equation.
In principle, the illposedness in the continuum model can also be fixed by keeping more terms in the asymptotic expansion from the discrete model. However, this will make the continuum model much more complicated.
The loss of accuracy of continuum models based on approximating inter-step distance by $a/|\nabla h|$ near the region $\nabla h=\mathbf 0$ has been noticed based on stepped surface in homoexpitaxy (i.e., without the long-range step interaction due to misfit) \cite{Israeli1999,E2001}. In \cite{Israeli1999}, the large terrace regions are separated from the vicinal regions described by the continuum model. In \cite{E2001}, it was proposed that dynamics of the tops and peaks of the surface where $h_x=0$ are governed separately as boundary conditions in the continuum model.

In order to fix this illposedness in the continuum model, we introduce a modified energy in the regime of small $|\nabla h|$ where $a/|\nabla h|$ no longer provides an approximation to the inter-step distance. As discussed above, the contribution of $\log|\nabla h|$ in the local step energy in Eq.~\eqref{eqn:g-line energy} takes a large negative value in this small $|\nabla h|$ limit. In fact, this energy contribution is an approximation of $-\log (l_t/a)$ in the discrete model where $l_t$ is the local inter-step distance.
Recall that the long-range step interaction energy due to misfit for a pair of steps separated by $l_t$ is proportional to $\log(l_t/a)$ (i.e., the interaction force is proportional to $a/l_t$) in the discrete model \cite{Tersoff1995Step,Duport1995New,Xiang2002Derivation,Xiang2004Misfit}.
That is, the illposedness or the negative local step energy of the continuum model in the small  $|\nabla h|$ regime is not physical, and it comes from the inappropriate continuum approximation of the inter-step distance in this regime in the local energy formula. We correct the local step energy by using $\log(|\nabla h|+\gamma_0)$ instead of $\log|\nabla h|$ in the continuum model, where $\gamma_0>0$ and $a/\gamma_0$ is a cutoff inter-step distance. The physical meaning of this modification by regularization is that when $|\nabla h|$ is too small and $a/|\nabla h|$ is not an approximation of the inter-step distance, we use a cutoff distance $a/\gamma_0$ to approximate the dimensionless inter-step distance in the continuum model.
The modified energy is:
\begin{align}\label{eq..ModifiedEnergy}
    E[h]
  &= -2c_1\pi^2L[\tilde{h}]_{H^{1/2}(\Omega)}^2+\int_{\Omega}\Psi(\nabla h)\diff{\vx},\\
  \Psi(\vp)
    &=ac_1\abs{\vp}\log(\abs{\vp}+\gamma_0)+ac_2\abs{\vp}+ac_3\abs{\vp}^3, \quad \vp\in\sR^2.\label{eq..phi}
\end{align}
The cutoff parameter $\gamma_0$ means that the maximum inter-step distance allowed in the continuum model is $a/\gamma_0$.
In practice,  the value of $\gamma_0$ can be chosen as a small positive constant that fixes the illposedness problem in the small $|\nabla h|$ regime, and in the meantime, only generates small error when $|\nabla h|$ is not that small.
Note that we do not need to modify other $\nabla h$ terms because they do not lead to physically irrelevant behavior in the small $|\nabla h|$ regime.

\begin{figure}[htbp]
\centering
	\subfigure[]{\includegraphics[width=0.49\textwidth]{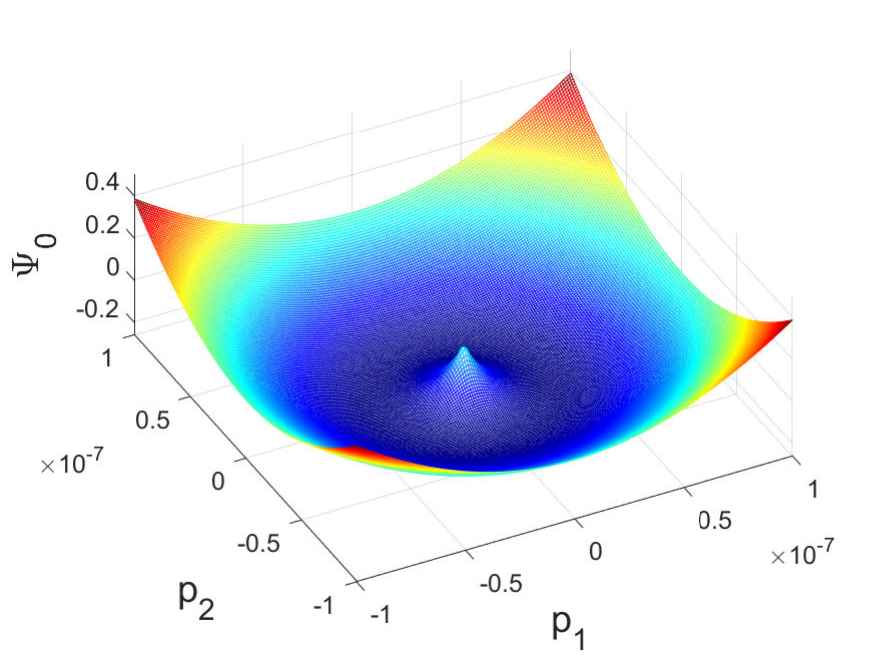}}
	\subfigure[]{\includegraphics[width=0.49\textwidth]{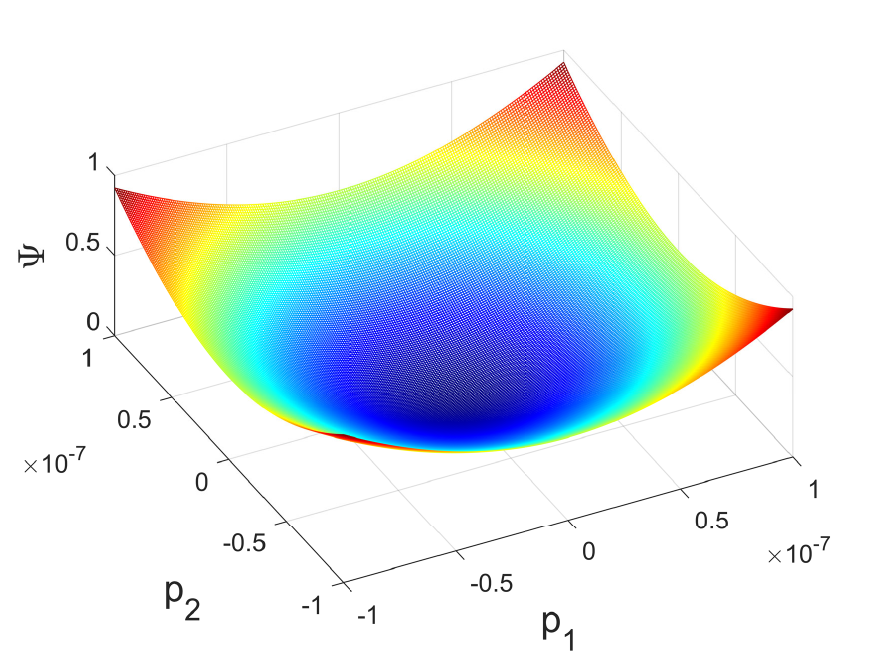}}
\caption{ (a) The original local energy density $\Psi_0(\vp)$ in the continuum model is non-convex. (b) The modified local energy density $\Psi(\vp)$ is convex under the condition \eqref{eqn:gamma0}. See \cref{paravalue} for values of parameters.}
\label{fig..modifytoconvex}
\end{figure}

We state that the modified local energy density $\Psi(\vp)$ is convex on the entire plane if the cutoff parameter $\gamma_0$ satisfies
$\gamma_0\geq \exp(-\frac{c_2}{c_1})$,
 while the original local energy density $\Psi_0(\vp)$ is not convex; see the illustration in \cref{fig..modifytoconvex} and the proof in \cref{prop..Convexity} below.  Moreover, under the condition $\gamma_0\geq \exp(-\frac{c_2}{c_1})$, we have  $\Psi(\vp)\geq a(c_1\log\gamma_0+c_2+c_3\abs{\vp}^2)\abs{\vp}\geq ac_3\abs{\vp}^3$, i.e., the generalized step line energy is always positive with the modified local energy.
It is easy for $\gamma_0$ to satisfy the condition $\gamma_0\geq \exp(-\frac{c_2}{c_1})$ because $\exp(-\frac{c_2}{c_1})$ in the physically meaningful regime is a very small number; see the calculation  in the remark below.
A possible choice of $\gamma_0$ is
\begin{equation}\label{eqn:gamma0}
 \gamma_0= \exp(-\frac{c_2}{c_1}),
 \end{equation}
  which is able to meet all these requirements. We use this value of $\gamma_0$ in this paper.
With this small value of  $\gamma_0$,
the difference between the modified local energy and original one is very small, and there is no significant change in the evolution of the surface; see the  comparisons in linear instability analysis and numerical simulation by using continuum models with and without this modification given in Appendix~\ref{appendix:regularization}.
This modification also does not affect the asymptotic behavior of the solution obtained in \cref{sec5}.

\begin{rmk}\label{paravalue}
  In practice, crystal grows on the vicinal surface which is cut at a small angle to the crystalline plane. For example, $g_1=\SI{0.03}{\joule\per \meter^2}$, $g_3=\SI{8.58}{\joule\per \meter^2}$, $a=\SI{0.27}{\nano\meter}$, $\nu=0.25$, $G=\SI{3.8e10}{\pascal}$, $r_c=a$, and $\eps_0=0.012$. In this case, $\sigma_0=\frac{2G(1+\nu)\eps_0}{1-\nu}=0.04G$, $c_1=\frac{(1-\nu)\sigma_0^2}{2\pi G}=\SI{7.2575e6}{\pascal}$, $c_2=\frac{g_1}{a}+c_1\log\frac{2\pi r_c}{\E a}=\SI{1.1719e8}{\pascal}$. Thus
  $\gamma_0=\exp(-\frac{c_2}{c_1})=\SI{9.7109e-8}{}$. Note that a typical miscut angle of the vicinal surface is of a few degree, and when the miscut angle $\theta=1^\circ$, the average slope of the surface $|\nabla h|=\tan\theta=\SI{1.75e-2}{}\gg \gamma_0$. Therefore, the modification with $\gamma_0$ only leads to a very small change in the  local energy.
\end{rmk}

\begin{rmk}
The modification that we made above in the continuum model is to fix the inaccurate formulation of the local energy in the small $|\nabla h|$ regime, which is based on the underlying physics. In its form, it is similar to a regularization of the logarithm term in the local energy.
Our modification is different from those regularizations in \cite{Gao2020analysis, Gao2019global, Gao2017weak}, whose purpose is to provide mathematical tools to overcome the challenges (e.g., singularities \cite{Gao2019global} and degeneracy \cite{Gao2017weak,Gao2020analysis}) when proving properties of the solutions, but these regularizations do not affect the physical properties.
\end{rmk}

\begin{rmk}
Note that unlike the 2+1 dimensional problem being considered in this paper, in the 1+1 dimensional continuum model \cite{Xiang2002Derivation,Xiang2004Misfit}, the corresponding local energy $\Psi_0(h_x)=ac_1h_x\log{h_x}+ac_2h_x+ac_3h_x^3$ is always convex for $h_x\ge 0$, since $\Psi_0''(h_x)=\frac{ac_1}{h_x}+6ac_3h_x\ge2a\sqrt{6c_1c_3}\ge 0$. There is no such nonconvexity and illposedness when the $1+1$ dimensional continuum model was analyzed \cite{DalMaso2014Analytical,Fonseca2015Regularity,Luo2017Bunching,Lu2018solutions}.
\end{rmk}

Convexity of the modified local energy $\Psi(\vp)$ and nonconvexity of the original local energy $\Psi_0(\vp)$ are shown in the following proposition:
\begin{prop}[Convexity]\label{prop..Convexity}
  The function $\Psi(\vp)$ given by Eq.~\eqref{eq..phi}  is convex on $\sR^2$ if $\gamma_0\geq \exp(-\frac{c_2}{c_1})$, and $\Psi_0(\vp)$ given by Eq.~\eqref{eq..energylocal}  is not convex on $\sR^2$.
\end{prop}
\begin{proof}
We first show that $\Psi(\vp)$ is convex.

  (1) Computing the Hessian $\nabla\nabla  \Psi$.

  On $\sR^2$, direct calculations of the derivatives of $\Psi$ lead to
  \begin{align*}
    \frac{\partial_{p_1 p_1}\Psi(\vp)}{a}
    =& \frac{c_1p_2^2}{\abs{\vp}^3}\log(\abs{\vp}+\gamma_0)+\frac{c_1}{\abs{\vp}+\gamma_0}+\frac{c_1\gamma_0p_1^2}{\abs{\vp}^2(\abs{\vp}+\gamma_0)^2}+\frac{c_2p_2^2}{\abs{\vp}^3}+3c_3\abs{\vp}+\frac{3c_3p_1^2}{\abs{\vp}},\\
    \frac{\partial_{p_1 p_2}\Psi(\vp)}{a}
    =& -\frac{c_1p_1p_2}{\abs{\vp}^3}\log(\abs{\vp}+\gamma_0)+\frac{c_1\gamma_0p_1p_2}{\abs{\vp}^2(\abs{\vp}+\gamma_0)^2}-\frac{c_2p_1p_2}{\abs{\vp}^3}+\frac{3c_3p_1p_2}{\abs{\vp}},\\
    \frac{\partial_{p_2 p_2}\Psi(\vp)}{a}
    =& \frac{c_1p_1^2}{\abs{\vp}^3}\log(\abs{\vp}+\gamma_0)+\frac{c_1}{\abs{\vp}+\gamma_0}+\frac{c_1\gamma_0p_2^2}{\abs{\vp}^2(\abs{\vp}+\gamma_0)^2}+\frac{c_2p_1^2}{\abs{\vp}^3}+3c_3\abs{\vp}+\frac{3c_3p_2^2}{\abs{\vp}}.
  \end{align*}

  (2) $\partial_{p_1p_1}\Psi\geq 0$ and $\partial_{p_2p_2}\Psi\geq 0$.

  When $\gamma_0\geq\exp(-\frac{c_2}{c_1})$,
  \begin{equation*}
     c_1\log(\abs{\vp}+\gamma_0)+c_2\geq 0, \quad i=1,2.
  \end{equation*}
  Therefore $\partial_{p_1p_1}\Psi\geq 0$ and $\partial_{p_2p_2}\Psi\geq 0$.

  (3) $\det\left(\nabla\nabla  \Psi\right)\geq 0$.
  \begin{align*}
    a^{-2}\det\left(\nabla\nabla  \Psi\right)
    =& a^{-2}\partial_{p_1 p_1}\Psi(\vp)\partial_{p_2 p_2}\Psi(\vp)-a^{-2}(\partial_{p_1 p_2}\Psi(\vp))^2\\
    =& \frac{c_1\gamma_0}{\abs{\vp}(\abs{\vp}+\gamma_0)^2}\left(c_1\log(\abs{\vp}+\gamma_0)+c_2\right)+6c_3\left(c_1\log(\abs{\vp}+\gamma_0)+c_2\right)\\
    &+\frac{c_1}{\abs{\vp}(\abs{\vp}+\gamma_0)}\left(c_1\log(\abs{\vp}+\gamma_0)+c_2\right)+\frac{c_1^2\gamma_0}{(\abs{\vp}+\gamma_0)^3}+\frac{3c_1c_3\gamma_0\abs{\vp}}{(\abs{\vp}+\gamma_0)^2}\\
    &+\frac{9c_1c_3\abs{\vp}}{\abs{\vp}+\gamma_0}+18c_3^2\abs{\vp}^2+\frac{c_1^2}{(\abs{\vp}+\gamma_0)^2}\\
    \geq& 0.
  \end{align*}

By (1)-(3),  $\Psi(\vp)$ is convex.

For $\Psi_0(\vp)$, we have
  \begin{align*}
    \frac{\partial_{p_1 p_1}\Psi_0(\vp)}{a}
    =& \frac{c_1p_2^2}{\abs{\vp}^3}\log\abs{\vp}+\frac{c_1}{\abs{\vp}}
    +\frac{c_2p_2^2}{\abs{\vp}^3}+3c_3\abs{\vp}+\frac{3c_3p_1^2}{\abs{\vp}}.
      \end{align*}
When $p_1=0$,
  \begin{align*}
    \frac{\partial_{p_1 p_1}\Psi_0(\vp)}{a}
    =& \frac{c_1}{\abs{p_2}}\log\abs{p_2}+\frac{c_1+c_2}{\abs{p_2}}
    +3c_3\abs{p_2},
      \end{align*}
which is negative when $p_2$ is small enough. Therefore, $\Psi_0(\vp)$ is not always convex.
\end{proof}

With the modified energy \eqref{eq..ModifiedEnergy}, the modified evolution equation is:
\begin{flalign}\label{eqn:evolution-m}
  h_t=&\nabla\cdot\left(D\nabla\mu\right),\\
  \mu=& -g_1\nabla\cdot\left(\frac{\nabla h}{\abs{\nabla h}}\right)-g_3\nabla\cdot\left(\abs{\nabla h}\nabla h\right)-\frac{(1-\nu)\sigma_0^2}{2\pi G}\int_{\sR^2}\frac{\vx-\vy}{\abs{\vx-\vy}^3}\cdot\nabla h(\vy)\diff{\vy} \nonumber \\
  &-\frac{(1-\nu)\sigma_0^2 a}{2\pi G}\left[\nabla\cdot\left(\frac{\nabla h}{\abs{\nabla h}}\right)\log\frac{2\pi r_c(\abs{\nabla h}+\gamma_0)}{a}+\frac{(\nabla h)^\T (\nabla\nabla h) \nabla h}{\abs{\nabla h}^2(\abs{\nabla h}+\gamma_0)}\right.\nonumber\\
  &\left.-\nabla\cdot\left( \frac{\gamma_0\nabla h}{\abs{\nabla h}(\abs{\nabla h}+\gamma_0)} \right)\right].\label{eqn:mu-m1}
\end{flalign}
We write the evolution equation as
\begin{flalign}\label{eq..evolution_equation}
    h_t
    =& -\Delta\left\{c_1\int_{\sR^2}\frac{\vx-\vy}{\abs{\vx-\vy}^3}\cdot\nabla h(\vy)\diff{\vy}+\nabla\cdot\zeta(\nabla h)
    \right\},\vspace{1ex}\\
    \zeta(\vp)=&\frac{ac_1\vp}{\abs{\vp}}\log (\abs{\vp}+\gamma_0)+\frac{ac_1\vp}{\abs{\vp}+\gamma_0} +\frac{ac_2\vp}{\abs{\vp}}+3ac_3\abs{\vp}\vp. \label{eq..zeta}
\end{flalign}

\subsection{$L^2$-gradient flow equation}
We will show the solution existence and uniqueness for an equivalent evolution equation with a vector-valued function $\vu:\Omega\times[0,+\infty)\to \sR^2$ rather than directly working on the equation of $h$. More precisely, we set
\begin{equation}\label{eq..handu}
h(\vx,t)=\nabla\cdot\vu(\vx,t)+\vB^{\T}\vx
\end{equation}
for all $\vx$ and $t$. Here $\vB^{\T}\vx$ is the reference plane mentioned in \cref{rmk..refernce}.
Note that  for a given $h$, such $\vu$ is unique up to a divergence free function.
The evolution equation of $h$ \eqref{eq..evolution_equation} is transformed equivalently to the evolution equation of $\vu$:
\begin{align}\label{eq..EvolutionEq.u}
    &\vu_t
    = -\nabla\left\{c_1\int_{\sR^2}\frac{\vx-\vy}{\abs{\vx-\vy}^3}\cdot \nabla \nabla\cdot\vu(\vy)\diff{\vy}
    +\nabla\cdot\zeta(\nabla \nabla\cdot\vu+\vB)\right\},
\end{align}
where $\zeta(\vp)$ is defined in Eq.~\eqref{eq..zeta}.

\begin{rmk}
    We take the initial value problem as an example to show the equivalence of the evolution equations of $h$ and $\vu$.
    From the relation between $h$ and $\vu$ in Eq.~\eqref{eq..handu}, we can take $\vu(x_1, x_2, t)$ as:
    \begin{equation*}
        \vu(x_1,x_2,t)={
        \left( \begin{array}{c}
        \displaystyle\frac12\int_0^{x_1}[\tilde{h}(s,x_2,t)-m_1(x_2,t)]\diff{s}+\frac{1}{2L}\int_0^L\int_0^{x_1}\tilde{h}(s,x_2,t)\diff{s}\diff{x_2}-n_1\\
        \displaystyle\frac12\int_0^{x_2}[\tilde{h}(x_1,s,t)-m_2(x_1,t)]\diff{s}+\frac{1}{2L}\int_0^L\int_0^{x_2}\tilde{h}(x_1,s,t)\diff{s}\diff{x_1}-n_2
        \end{array}
        \right )},
    \end{equation*}
    where $m_1(x_2,t)=\frac{1}{L}\int_0^L\tilde{h}(x_1,x_2,t)\diff{x_1}$ and $m_2(x_2,t)=\frac{1}{L}\int_0^L\tilde{h}(x_1,x_2,t)\diff{x_2}$
    guarantee the periodicity of $\vu$ on $\Omega$, $n_i$ are constants that adjust the integral value of $\vu$, for $i=1,2$. Then, for a given initial value $h(x_1,x_2,0)$, we obtain $\vu(x_1,x_2,0)$ by substituting $h(x_1,x_2,0)$ to the above form. And for a given initial value $\vu(x_1,x_2,0)$, we obtain $h(x_1,x_2,0)$ from \eqref{eq..handu}.
\end{rmk}

The total energy associated with $\vu$ is
\begin{equation}\label{eq..Energy.u}
    F[\vu]
    = -\frac{c_1}{2}\int_\Omega \nabla\cdot \vu(\vx)\int_{\sR^2}\frac{\vx-\vy}{\abs{\vx-\vy}^3}\cdot\nabla\nabla\cdot \vu(\vy)\diff{\vy}\diff{\vx}+\int_{\Omega}\Psi(\nabla\nabla\cdot\vu+\vB)\diff{\vx}.
\end{equation}
Note that $F[\vu]=E[h]$ and the evolution equation of $\vu$ is $L^2$-gradient flow:
\begin{equation}\label{eq..L2u}
    \vu_t=-\frac{\delta F}{\delta \vu}.
\end{equation}

    The evolution equation of profile height \eqref{eq..evolution_equation} is a $H^{-1}$-gradient flow. By transforming $h(\vx)$ to $\vu(\vx)$, we obtain the $L^2$-gradient flow equation. Since the theory of gradient flows in Hilbert spaces is well developed, there are more tools in analysing equation of $\vu$. Note that the available theoretical analyses \cite{DalMaso2014Analytical,Fonseca2015Regularity,Gao2020regularity} for equations in 1+1 dimensions were based on this kind of transform.

Our main functional space will be $L_{\#0}^2(\Omega)$, in which the functions are square integrable, periodic, with zero average, and $W_{\#0}^{k,p}$, in which the functions belong to $W_{\#}^{k,p}$ with zero average.
Let $D(F):=\{\vu\in L^2(\Omega)\mid F[\vu]<+\infty\}$ be the set of functions with finite energy, and $\overline{D(F)}^{\norm{\cdot}_{L^2(\Omega)}}$ be the closure of $D(F)$ with respect to the $L^2$ distance. Note that a function $\vu$ with $\nabla\nabla\cdot\vu+\vB=0$ on some non-negligible set also exists in $D(F)$. However, such a function has no classical variation $\frac{\delta F}{\delta \vu}$ due to the singularities; see the definition of $\zeta(\vp)$ in Eq.~\eqref{eq..zeta}. Thus we should firstly consider sub-gradients and prove the existence result of evolution variational inequality
\begin{equation}\label{eq..EVI}
    \vu_t\in-\partial F[\vu]\quad \text{for a.e.}\quad t>0.
\end{equation}
We say $\vu$ is a variational inequality solution if $\vu(t)$ is a locally absolutely continuous curve such that $\lim_{t\to0}\vu(t)=\vu_0$ in $L^2(\Omega)$ for given initial datum $\vu_0$ and
\begin{equation}\label{eq..solEVI}
    \frac12\frac{\diff}{\diff{t}}\|\vu(t)-\vv\|_{L^2(\Omega)}^2\le F(\vv)-F(\vu(t)) \quad\text{a.e.}~t>0,\forall \vv\in D(F).
\end{equation}

\section{Main analysis results}\label{sec:main}

Now we state our main analysis results of existence and uniqueness of the energy minimizer  and the weak solution of evolution variational inequality
 as well as the energy scaling law in 2+1 dimensions.

\begin{thm}[Existence and uniqueness of energy minimizer]\label{thm..ExistenceEnergyMinimizer}
   Given the domain $\Omega$ and the average slope $\vB$, there exists a global minimizer $h^*$ of energy \eqref{eq..ModifiedEnergy} in the solution space $X$, that is
   \begin{equation}
        E[h^*]=\min\limits_{h\in X}E[h].
   \end{equation}
   Moreover, if the ratio of the domain size $L$ and the lattice constant $a$ satisfies $\frac{L}{a}<\beta$, where
    \begin{equation}\label{eq..coeffPi}
    \beta=\left\{
    \begin{aligned}
    &2\sqrt{\frac{3c_3}{c_1}}-\frac{3c_3}{c_1}\gamma_0&\quad \sqrt{\frac{c_1}{3c_3}}\ge\gamma_0,\\
    &\frac{1}{\gamma_0}&\quad \text{otherwise},
    \end{aligned}
    \right.
    \end{equation}
    then the minimizer $h\in X$ of energy \eqref{eq..ModifiedEnergy} is unique.
\end{thm}

\begin{thm}[Existence and uniqueness of weak solution of evolution variational inequality]\label{thm..ExistenceEVI}
    Suppose the ratio of the domain size $L$ and the lattice constant $a$ satisfies $\frac{L}{a}<\beta$, where $\beta$ takes as \eqref{eq..coeffPi}
    and for any initial data $\vu_0\in\overline{D(F)}^{\norm{\cdot}_{L^2(\Omega)}}$, there exists a unique solution $\vu$ satisfying the evolution variational inequality \eqref{eq..EVI} and
    $\vu\in L^\infty_\mathrm{loc}(0,+\infty;W^{2,3}_{\#0}(\Omega))$,
    $\vu_t\in L^\infty(0,+\infty;L^2(\Omega))$.
\end{thm}

\begin{rmk}
Except for the existence of energy minimizer, the conclusions in the above two theorems rely on the assumption $L/a<\beta$, where $L$ is the domain size.  The physical meaning of this assumption is that the stabilizing effect of the force dipole interaction, i.e., the $c_3$ (or $g_3$) term in \eqref{eq..ModifiedEnergy} and \eqref{eq..phi}, should dominate over the destabilizing long-range effect of the misfit energy, i.e., the term with $H^\frac{1}{2}$ norm in \eqref{eq..ModifiedEnergy}. From Eq.~\eqref{eq..EnergyPeriodic}, it can be seen
that the misfit energy that has the  destabilizing long-range effect  is proportional to $L^3$ whereas all the local energies including that due to the stabilizing  force dipole interaction is only proportional to $L^2$, which makes the proofs challenging. This challenge also exists in the continuum model in $1+1$ dimension \cite{Xiang2002Derivation}. The available analyses  \cite{DalMaso2014Analytical,Fonseca2015Regularity,Gao2017continuum}  for the $1+1$ dimensional model were based on the special case where all the coefficients in the PDE are equal to $1$ and the domain has a fixed $O(1)$ size ($2\pi$ or $1$), and in this special case, the above mentioned assumption is satisfied.

For the range of physically meaning values of $\beta$,
 we consider parameters $a=\SI{0.27}{\nano\meter}$, $\nu=0.25$, $G=\SI{3.8e10}{\pascal}$, $r_c=a$, and $\eps_0=0.012$.
    When the parameters $g_1=\SI{0.03}{\joule\per \meter^2}$ and $g_3=\SI{8.58}{\joule\per \meter^2}$ as in Refs.~\cite{shenoy2002continuum,Zhu2009Continuum},
    we have   $\beta=130$.
      From Refs.~\cite{jeong1999steps,sudoh1998step}, for Si(113) at $\SI{983}{\K}$, we have $g_1=\SI{0.3382}{\joule\per \meter^2}$, $g_3=\SI{ 5767.8}{\joule\per \meter^2}$, and $\beta=3431$. From Ref.~\cite{jeong1999steps}, for Si(111) at $\SI{1223}{\K}$, we have $g_1=\SI{0.1778}{\joule\per \meter^2}$, $g_3=\SI{0.8011}{\joule\per \meter^2}$, under which $\beta=40$. These values show that the assumption $L/a<\beta$ holds for  reasonable sizes of domain for the cases of strong force dipole interaction and/or small misfit.
\end{rmk}

\begin{thm}[Energy scaling law]\label{thm..scalinglaw}
Given the domain $\Omega$ and the average slope $\vB$, the following energy scaling law holds for energy \eqref{eq..ModifiedEnergy} with some positive constants $C_1, C_2$
\begin{equation}\label{eqn:scalinglaw}
    -C_1a^{-2}\le\inf\limits_{h\in X} E[h]\le-C_2a^{-2}, \ \ a\rightarrow 0.
\end{equation}
\end{thm}

\begin{rmk}
Recall that in Ref.~\cite{Luo2017Bunching}, the energy functional for 1+1 dimensional continuum model
\begin{equation}\label{eqn:energy1d}
    E[h]=-\frac12\int_{\Omega} h(x)\mathrm{P.V.}\int_{\sR}\frac{h_x(y)}{x-y}\diff{y}\diff{x}+a\int_{\Omega}\left(h_x\log{h_x}+\frac{\gamma}{6}h_x^3\right)\diff{x},
\end{equation}
has a global minimizer in $X^{'}$ and the following energy scaling law
\begin{equation}\label{eqn:scaling1d}
    \frac{B^2}{2}\log{a}-C\le\inf\limits_{h\in X^{'}}E[h]\le \frac{B^2}{2}\log{a}+C', \ \ a\rightarrow 0,
\end{equation}
if the average slope $B>0$, for some positive constants $C,C'$. Here $\Omega=[-\frac12,\frac12]$ is a periodic cell and the solution space $X^{'}=\{h\in H_{loc}^1(\sR):
\tilde{h}(x)=h(x)-Bx\in H^1(\sR)$ with $\Omega$-period weak derivatives, $\int_{\Omega}\tilde{h}(x)\diff{x}=0
,h_x\ge 0,a.e.~ x\in\sR\}$.
Notice that the energy scaling in 2+1 dimensions in Eq.~\eqref{eqn:scalinglaw} is essentially different from that  in 1+1 dimensions in Eq.~\eqref{eqn:scaling1d}. As will be shown in \cref{sec5}, this essential difference is due to the fact that the step meandering instability dominates in 2+1 dimensions, whereas the step bunching instability dominates in 1+1 dimensions.
\end{rmk}

\section{Existence and uniqueness of energy minimizer}\label{sec3}

In this section, we use the direct method in the calculus of variations to prove the minimizer existence and uniqueness of the energy \eqref{eq..ModifiedEnergy}.

Weak lower semi-continuity  is commonly used in the  existence proof of the minimizer of the energy functional. However, the standard weak lower semi-continuity \cite{Evans2010Partial}   cannot be applied directly to our problem because our total energy  \eqref{eq..ModifiedEnergy} is not convex in $\nabla h$  due to the presence of the negative contribution of the $H^{1/2}$ semi-norm  term. We will write another version of weak lower semi-continuity and then use it to prove the existence and uniqueness result together with coercivity of $E[h]$ given below and convexity of $\Psi(\vp)$ in \cref{prop..Convexity}.

Before the proof, we rephrase the standard weak lower semi-continuity \cite{Evans2010Partial} for energy functional as follows.
\begin{prop}[\cite{Evans2010Partial} Standard weak lower semi-continuity]\label{prop..WeakLowerSemicont}
  Assume that $\Psi$ is bounded below and the mapping $\vp\mapsto \Psi(\vp, z, \vx)$ is convex, for each $z\in\sR$, $\vx\in \Omega$. Then $E[h]:=\int_{\Omega}\Psi(\nabla h(\vx), h(\vx), \vx)\diff{\vx}$ is weakly lower semi-continuous.
\end{prop}

\begin{prop}[Coercivity]\label{prop..Coercivity}
  Suppose that $\tilde{h}\in V$. Then there exists a constant $C$ (depending on the coefficients $a$, $c_1$, and $c_3$ in \eqref{eq..phi}), such that
  \begin{equation*}
    E[h]\geq \frac{c_1}{2}L\norm{\nabla\tilde{h}}_{L^2(\Omega)}^2-CL^2.
  \end{equation*}
\end{prop}
\begin{proof}
  Since $\tilde{h}\in V$, we have
  \begin{equation}\label{nonlocal-energy-2}
  [\tilde{h}]_{H^{1/2}(\Omega)}^2=\sum\limits_{\vk\in\sZ^2}\abs{\vk}\abs{h_{\vk}}^2\leq\sum
  \limits_{\vk\in\sZ^2}\abs{\vk}^2\abs{h_{\vk}}^2=\frac{1}{4\pi^2}\norm{\nabla\tilde{h}}_{L^2(\Omega)}^2.
  \end{equation}
  Note that we have $c_1\abs{\nabla h}\log(\abs{\nabla h}+\gamma_0)+c_2\abs{\nabla h}\geq 0$ when $\gamma_0=\exp(-\frac{c_2}{c_1})$. Thus
  \begin{align}
    E[h]
    &\geq -\frac{c_1}{2}L\norm{\nabla \tilde{h}}_{L^2(\Omega)}^2+ac_3\int_{\Omega}\abs{\nabla h}^3\diff{\vx}\label{coer-E}\\
    &= \int_{\Omega}\left\{ac_3\abs{\nabla h}^3-c_1 L \abs{\nabla\tilde{h}}^2+C\right\}\diff{\vx}+\frac{c_1}{2} L\norm{\nabla\tilde{h}}_{L^2(\Omega)}^2-CL^2.
  \end{align}
  Note that $\nabla h=\nabla \tilde{h}+\vB$. Choose $C=-\min\limits_{\abs{\nabla \tilde{h}}\in\sR^2}\left\{ac_3\abs{\nabla h}^3-c_1 L \abs{\nabla\tilde{h}}^2\right\}\\=-\min\limits_{\vp\in\sR^2}\left\{ac_3\abs{\vp+\vB}^3-c_1 L \abs{\vp}^2\right\}<+\infty$. The required lower bound holds.
\end{proof}

\begin{prop}[Weak lower semi-continuity]\label{prop..SequentialWeakLowerSemicont}
  Suppose there is a sequence $\{h^k\}_{k=1}^{\infty}\subset X$ and the weak convergence $h^k\rightharpoonup h\in X$ holds as $k\to+\infty$. Then there exists a subsequence $\{h^{k_j}\}_{j=1}^{\infty}\subset\{h^k\}_{k=1}^{\infty}$ such that $\liminf\limits_{j\to+\infty}E[h^{k_j}]\geq E[h]$.
\end{prop}

\begin{proof}
  1. We split the energy into two parts $E[h]=E_1[h]+E_2[h]$, where
  \begin{align*}
    E_1[h]
    &= -2c_1\pi^2 L [\tilde{h}]_{H^{1/2}(\Omega)}^2,\\
    E_2[h]
    &= \int_{\Omega}\Psi(\nabla h)\diff{\vx}
  \end{align*}
  Applying compact Sobolev embedding theorem $H^1(\Omega)\hookrightarrow\hookrightarrow H^{1/2}(\Omega)$, the weak convergence $\tilde{h}^k\rightharpoonup \tilde{h}\in H^1(\Omega)$ implies that there is a strong convergence subsequence $\tilde{h}^{k_j}\to \tilde{h}\in H^{1/2}(\Omega)$. Therefore, by passing to a subsequence $\{h^{k_j}\}_{j=1}^{\infty}$,
  \begin{equation}
    \liminf\limits_{j\to+\infty}E_1[h^{k_j}]=E_1[h].
  \end{equation}
  We now only need to show that $\liminf\limits_{j\to+\infty}E_2[h^{k_j}]\geq E_2[h].$

  2. By \cref{prop..Convexity}, the mapping $\vp\to \Psi(\vp)$ is convex on $\sR^2$.

  3. By \cref{prop..Coercivity}, $E[h]\geq \frac{c_1}{2}L \norm{\nabla\tilde{h}}_{L^2(\Omega)}^2-CL^2$. Note that $E_1[h]\leq 0$, we have $E_2[h]=E[h]-E_1[h]\geq E[h]$ is lower bounded.

  4. Based on the convexity of $\vp\mapsto \Psi(\vp)$ and coercivity of the second part $E_2[h]\geq \frac{c_1}{2}L\norm{\nabla\tilde{h}}_{L^2(\Omega)}^2-CL^2$, we apply the usual weak lower semi-continuity result (\cref{prop..WeakLowerSemicont}) to subsequence $\{h^{k_j}\}_{j=1}^{\infty}$ and energy functional $E_2[h]$ to find $\liminf\limits_{j\to+\infty}E_2[h^{k_j}]\geq E_2[h]$. This completes the proof.
\end{proof}

Next we prove the existence of the energy minimizer with the coercivity and lower semi-continuity by the direct method in the calculus of variations.

\begin{proof}[Proof of \cref{thm..ExistenceEnergyMinimizer} $($Existence$)$]
  Let $m:=\inf\limits_{h\in X}E[h]$. Note that the flat surface $\vB^\T\vx\in X$ and $E[\vB^\T \vx]=0+aL^2\left(c_1\abs{\vB}\log({\abs{\vB}}+\gamma_0)+c_2\abs{\vB}+c_3\abs{\vB}^3\right)$. Thus $m\leq \inf\limits_{\vB\in\sR^2} E[\vB^\T\vx]<+\infty$. By \cref{prop..Coercivity}, $m\geq-CL^2>-\infty$. Hence, $m$ is finite. Now select a minimizing sequence $\{h^k\}_{k=1}^{\infty}\subset X$ and $\tilde{h}^k(\vx)=h^k(\vx)-\vB^\T \vx$ with $E[h^k]\to m$. By \cref{prop..Coercivity} again, $E[h^k]\geq \frac{c_1}{2}L \norm{\nabla\tilde{h}^k}_{L^2(\Omega)}^2-CL^2$. And since $E[h^k]\to m$, we conclude that $\sup\limits_{k}\norm{\nabla\tilde{h}^k}_{L^2(\Omega)}<+\infty$. By Poincare inequality, $\norm{\tilde{h}^k}_{L^2(\Omega)}\leq C\norm{\nabla \tilde{h}^k}_{L^2(\Omega)}$. Hence $\sup\limits_{k}\norm{\tilde{h}^k}_{L^2(\Omega)}<+\infty$. These estimates imply that $\{\tilde{h}^k\}_{k=1}^{\infty}$ is bounded in $V$.

  Consequently, there exist a subsequence $\{h^{k_j}\}_{j=1}^{\infty}\subset\{h^k\}_{k=1}^{\infty}$ and a function $\tilde{h}^*\in V$ such that
  $\tilde{h}^{k_j}\rightharpoonup \tilde{h}^*$ weakly in $V$. Note that $X$ is a convex, closed subset of $V$. Then $X$ is weakly closed due to Mazur's Theorem. Thus $h^*\in X$.
  By \cref{prop..SequentialWeakLowerSemicont}, there exists a further subsequence, we still denote it as $\{h^{k_j}\}_{j=1}^{\infty}$, such that $E[h^*]\leq\liminf\limits_{j\to+\infty}E[h^{k_j}]$. Note that $\left\{E[h^{k}]\right\}_{k=1}^{\infty}$ converges to $m$, so as $\left\{E[h^{k_j}]\right\}_{j=1}^{\infty}$. It follows that $E[h^*]=m=\min\limits_{h\in X}E[h]$.
\end{proof}

To ensure the uniqueness we introduce the following Proposition on the strict convexity of local energy density.
\begin{prop}[Strict convexity]\label{prop..unifconvex}
The mapping $\vp\mapsto\Psi(\vp)$ satisfies
\begin{equation}
    \sum_{i,j=1}^2\Psi_{p_ip_j}(\vp)\xi_i\xi_j\ge ac_1\beta\abs{\vxi}^2, \quad \vp,\vxi\in\sR^2,
\end{equation}
where $\beta>0$ is defined in \eqref{eq..coeffPi}.
\end{prop}
\begin{proof}
Define $\Phi(\vp):= \Psi(\vp)-\frac{1}{2}ac_1\beta\abs{\vp}^2$, we need to show $\Phi(\vp)$ is convex on $\sR^2$.\\
(1) Computing $\nabla\nabla  \Phi$.\\
  On $\sR^2$, direct calculations of the derivatives of $\Phi$ lead to
 \begin{equation*}
    \partial_{p_i p_j}\Phi(\vp)
    = \partial_{p_i p_j}\Psi(\vp)-ac_1\beta \delta_{ij},\quad i,j=1,2
  \end{equation*}
  where $\delta_{ij}$ is Kronecker delta function and $\partial_{p_i p_j}\Psi(\vp)$ is calculated in \cref{prop..Convexity}.\\
  (2) $\partial_{p_1p_1}\Phi\geq 0$ and $\partial_{p_2p_2}\Phi\geq 0$.\\
Note that
  \begin{align*}
    a^{-1}\partial_{p_1 p_1}\Phi(\vp)
    &\ge c_1\left[\frac{1}{\abs{\vp}+\gamma_0}+\frac{\gamma_0p_1^2}{\abs{\vp}^2(\abs{\vp}+\gamma_0)^2}\right]
    +3c_3\left[\abs{\vp}+\frac{p_1^2}{\abs{\vp}}\right]-c_1\beta,\\
    &\ge \frac{c_1}{\abs{\vp}+\gamma_0}+3c_3(\abs{\vp}+\gamma_0)-3c_3\gamma_0-c_1\beta,\\
    &\ge \left\{
    \begin{aligned}
    &2\sqrt{3c_1c_3}-3c_3\gamma_0-c_1\beta& \quad \sqrt{\frac{c_1}{3c_3}}\ge\gamma_0\\
    &\frac{c_1}{\gamma_0}-c_1\beta& \quad \text{otherwise}.
    \end{aligned}
    \right.
  \end{align*}
Thus, for $\beta$ taken as \eqref{eq..coeffPi}, we have $\partial_{p_1 p_1}\Phi(\vp)\ge 0$ for all $\vp\in \sR^2$. Similarly, we can prove $\partial_{p_2 p_2}\Phi(\vp)\ge 0$ for all $\vp\in \sR^2$.\\
(3) $\det\left(\nabla\nabla  \Phi\right)\geq 0$.
 \begin{align*}
    a^{-2}\det\left(\nabla\nabla \Phi\right)
    =& \left(a^{-1}\partial_{p_1 p_1}\Psi(\vp)-c_1\beta\right)\left(a^{-1}\partial_{p_2 p_2}\Psi(\vp)-c_1\beta\right)-a^{-2}(\partial_{p_1 p_2}\Psi(\vp))^2\\
    =&c_1^2\beta^2-\left[\frac{c_1}{\abs{\vp}}\log(\abs{\vp}+\gamma_0)+\frac{2c_1}{\abs{\vp}+\gamma_0}
    +\frac{c_1\gamma_0}{(\abs{\vp}+\gamma_0)^2}+\frac{c_2}{\abs{\vp}}+9c_3\abs{\vp}\right]c_1\beta\\
    &+\frac{c_1\gamma_0}{\abs{\vp}(\abs{\vp}+\gamma_0)^2}\left(c_1\log(\abs{\vp}+\gamma_0)+c_2\right)+6c_3\left(c_1\log(\abs{\vp}+\gamma_0)+c_2\right)\\
    &+\frac{c_1}{\abs{\vp}(\abs{\vp}+\gamma_0)}\left(c_1\log(\abs{\vp}+\gamma_0)+c_2\right)+\frac{c_1^2\gamma_0}{(\abs{\vp}+\gamma_0)^3}+\frac{3c_1c_3\gamma_0\abs{\vp}}{(\abs{\vp}+\gamma_0)^2}\\
    &+\frac{9c_1c_3\abs{\vp}}{\abs{\vp}+\gamma_0}+18c_3^2\abs{\vp}^2+\frac{c_1^2}{(\abs{\vp}+\gamma_0)^2}.
 \end{align*}
For the quadratic function of $\beta$ in the above equation, the discriminant is
\begin{equation*}
    \Delta
    =c_1^2\left(\frac{c_1\gamma_0}{(\abs{\vp}+\gamma_0)^2}+3c_3\abs{\vp}-\frac{1}{\abs{\vp}}(c_1\log(\abs{\vp}+\gamma_0)+c_2)\right)^2.
\end{equation*}
Then, the smaller real root $\beta_{-}$ can be expressed as
\begin{align*}
    2\beta_{-}=&\frac{1}{\abs{\vp}}\log(\abs{\vp}+\gamma_0)+\frac{2}{\abs{\vp}+\gamma_0}
    +\frac{\gamma_0}{(\abs{\vp}+\gamma_0)^2}+\frac{c_2}{c_1\abs{\vp}}+9\frac{c_3}{c_1}\abs{\vp}-\frac{\sqrt{\Delta}}{c_1}\\
    =&\frac{2}{\abs{\vp}}\log(\abs{\vp}+\gamma_0)+\frac{2c_2}{c_1\abs{\vp}}
    +\frac{2}{\abs{\vp}+\gamma_0}+6\frac{c_3}{c_1}\abs{\vp}\\
    \ge& \frac{2}{\abs{\vp}+\gamma_0}+6\frac{c_3}{c_1}(\abs{\vp}+\gamma_0)-6\frac{c_3}{c_1}\gamma_0\\
    \ge&\left\{
        \begin{aligned}
        &2\left[2\sqrt{\frac{3c_3}{c_1}}-3\frac{c_3}{c_1}\gamma_0\right],&\quad \sqrt{\frac{c_1}{3c_3}}\ge\gamma_0\\
        &\frac{2}{\gamma_0},&\quad \text{otherwise}
        \end{aligned}
    \right.
\end{align*}

Thus, when $\beta$ takes \eqref{eq..coeffPi}, we obtain that $\forall \vp\in \sR^2, \beta_{-}\ge\beta$ and the quadratic expression of $\beta$ is positive for $\vp\in\sR^2$. Therefore, $\det\left(\nabla\nabla  \Phi\right)\geq 0$.
\end{proof}

\begin{proof}[Proof of  \cref{thm..ExistenceEnergyMinimizer} $($Uniqueness$)$]
    Assume $h_1, h_2\in X$ are both minimizers of $E[h]$ over $X$. Then, $h_3:=\frac{h_1+h_2}{2}\in X$. We claim that
    \begin{equation}\label{uniq}
        E[h_3]\le\frac{ E[h_1]+ E[h_2]}{2},
    \end{equation}
    with a strict inequality, unless $h_1=h_2$ a.e.

    Denote $I[h]:=-\frac{c_1}{2}\int_{\Omega} h(\vx)\int_{\sR^2}\frac{\vx-\vy}{\abs{\vx-\vy}^3}\cdot\nabla h(\vy)\diff{\vy}\diff{\vx}$, we have $E[h]=I[h]+\int_\Omega\Psi(\nabla h)\diff{\vx}$.
    By direct calculation, we have
    \begin{equation}\label{uniqnonlocal}
        I[h_3]=\frac{I[h_1]+I[h_2]}{2}+\frac{c_1}{8}\int_{\Omega} (h_1-h_2)(\vx)\int_{\sR^2}\frac{\vx-\vy}{\abs{\vx-\vy}^3}\cdot\nabla (h_1-h_2)(\vy)\diff{\vy}\diff{\vx}
    \end{equation}
    Note from the strict convexity of $\Psi(\vp)$ in Proposition \ref{prop..unifconvex} that
    \begin{equation*}
        \Psi(\vp)\ge\Psi(\vq)+\nabla\Psi(\vq)\cdot(\vp-\vq)+\frac{ac_1\beta}{2}|\vp-\vq|^2,\quad \vp,\vq\in \sR^2.
    \end{equation*}
    Setting $\vq=\nabla h_3$ and $\vp=\nabla h_1,\nabla h_2$, respectively, and then integrating over $\Omega$ and add them up, we have:
    \begin{equation}\label{uniqlocal}
        \int_\Omega\Psi(\nabla h_3)\diff{\vx}+\frac{ac_1\beta}{8}\int_\Omega|\nabla h_1-\nabla h_2|^2\diff{\vx}\le \frac{\int_\Omega\Psi(\nabla h_1)\diff{\vx}+\int_\Omega\Psi(\nabla h_2)\diff{\vx}}{2}.
    \end{equation}
    Add up \eqref{uniqnonlocal} and \eqref{uniqlocal}, we obtain
    \begin{align*}
        E[h_3]+\frac{ac_1\beta}{8}\int_\Omega|\nabla h_1-\nabla h_2|^2\diff{\vx}&\le
        \frac{E[h_1]+E[h_2]}{2}+\frac{c_1L}{8}\int_\Omega|\nabla h_1-\nabla h_2|^2\diff{\vx},
    \end{align*}
    where the inequality comes from $H^{1/2}$ norm estimation.
    Therefore,
    \begin{equation}
        E[h_3]+\frac{ac_1\beta-c_1L}{8}\int_\Omega|\nabla h_1-\nabla h_2|^2\le \frac{E[h_1]+E[h_2]}{2}.
    \end{equation}
    The assumption $\frac{L}{a}<\beta$ guarantees \eqref{uniq}.

    As $E[h_1]=E[h_2]=\min\limits_{h\in X} E[h]\le E[h_3]$, we deduce $\nabla h_1=\nabla h_2$ a.e. in $\Omega$. Since both $h_1$ and $h_2$ have averaged slope $\vB$, it follows that
    $h_1=h_2$ a.e. in $\Omega$.
\end{proof}

\section{Existence and uniqueness of weak solution for evolution equation}\label{sec4}
In this section, we prove the existence and uniqueness of the weak solution of the evolution equation \eqref{eq..EVI}. In the proof, we will use the following proposition and the framework of gradient flow analysis in Ref.~\cite{Ambrosio2008Gradient}.

\begin{prop}\label{prop..BoundednessConvexityLSCCompactness}
    With the ratio of the domain length $L$ and the lattice constant $a$ satisfying $\frac{L}{a}<\beta$, where $\beta$ is defined in \eqref{eq..coeffPi}, the energy $F$ is $\lambda$-convex with $\lambda=ac_1\beta-c_1L$ in the $L^2$-topology and lower semi-continuous with respect to the weak $L^2$-topology. Moreover, the sub-levels of $F$ are compact in the $L^2$-topology. (See Appendix~\ref{appendix:convexity} for the definition of $\lambda$-convexity~\cite{Ambrosio2008Gradient}.)
\end{prop}
\begin{proof}
    1. (Boundedness from below) Since $\nabla\cdot \vu\in V$, $\left[\nabla\cdot \vu \right]_{H^{1/2}(\Omega)}^2\le\frac{1}{4\pi^2}\|\nabla\nabla\cdot \vu\|_{L^2(\Omega)}^2$.
    Note that $c_1\abs{\nabla\nabla\cdot \vu+\vB}\log(\abs{\nabla\nabla\cdot \vu+\vB}+\gamma_0)+c_2\abs{\nabla\nabla\cdot \vu+\vB}\geq 0$, recalling that $\gamma_0=\exp(-\frac{c_2}{c_1})$. We have
    \begin{align}
        F[\vu]&=-2c_1\pi^2L\left[\nabla\cdot \vu \right]_{H^{1/2}(\Omega)}^2+\int_\Omega\Psi(\nabla\nabla\cdot \vu+\vB)\diff{\vx}\nonumber\\
        &\ge -\frac{c_1L}{2}\|\nabla\nabla\cdot \vu\|_{L^2(\Omega)}^2+ac_3\|\nabla\nabla\cdot \vu+\vB\|_{L^3(\Omega)}^3\nonumber\\
        &=\int_{\Omega}\left\{ac_3\abs{\nabla\nabla\cdot \vu+\vB}^3-c_1L\abs{\nabla\nabla\cdot \vu}^2+C \right\}\diff{\vx}+\frac{c_1 L}{2}\int_{\Omega}\abs{\nabla\nabla\cdot\vu}^2\diff{\vx}-CL^2\nonumber\\
        &\ge\frac{c_1 L}{2}\|\nabla\nabla\cdot\vu\|_{L^2(\Omega)}^2-CL^2,\nonumber
    \end{align}
    where $C=-\min\limits_{\abs{\nabla\nabla\cdot\vu}}\left\{ac_3\abs{\nabla\nabla\cdot \vu+\vB}^3-c_1L\abs{\nabla\nabla\cdot \vu}^2\right\}<+\infty$. Thus, the energy $F[\vu]$ is bounded from below.

   2. ($\lambda$-convexity)
 Rewrite the total energy as
    \begin{equation*}
        F[\vu]=F_1[\vu]+F_2[\vu]+F_3[\vu],
    \end{equation*}
    where
    \begin{align*}
        F_1[\vu]=&-\frac{c_1}{2}\int_\Omega \nabla\cdot \vu(\vx)\int_{\sR^2}\frac{\vx-\vy}{\abs{\vx-\vy}^3}\cdot\nabla\nabla\cdot \vu(\vy)\diff{\vy}\diff{\vx},\\
        F_2[\vu]=&a\frac{c_1\beta}{2}\|\nabla\nabla\cdot\vu+\vB\|_{L^2(\Omega)}^2,\\
        F_3[\vu]=&\int_\Omega \left\{\Psi(\nabla\nabla\cdot\vu+\vB)-a\frac{c_1\beta}{2}\abs{\nabla\nabla\cdot\vu+\vB}^2\right\}\diff{\vx}.
    \end{align*}

        Given $\vu,\vv\in D(F)$, $s\in[0,1]$, for $F_1[\vu]$, using Eqs.~\eqref{nonlocal-energy-1} and \eqref{nonlocal-energy-2}, we have
    \begin{align*}
        F_1[s\vu+(1-s)\vv]&= s F_1[\vu]+(1-s)F_1[\vv]\\
        +&\frac{s(1-s)c_1}{2}\int_\Omega \nabla\cdot (\vu-\vv)(\vx)\int_{\sR^2}\frac{\vx-\vy}{\abs{\vx-\vy}^3}\cdot\nabla\nabla\cdot (\vu-\vv)(\vy)\diff{\vy}\diff{\vx}\\
        &\le s F_1[\vu]+(1-s)F_1[\vv]-s(1-s)\frac{-c_1L}{2}\|\nabla\nabla\cdot(\vu-\vv)\|_{L^2(\Omega)}^2.
        \end{align*}
  For $F_2[\vu]$, we have
        \begin{align*}
        F_2[s\vu+(1-s)\vv]&=sF_2[\vu]+(1-s)F_2[\vv]-s(1-s)\frac{ac_1\beta}{2}\|\nabla\nabla\cdot(\vu-\vv)\|_{L^2(\Omega)}^2.
    \end{align*}
  We know from \cref{prop..unifconvex} that $F_3[\vu]$ is convex, thus,  for $s\in[0,1]$,
    \begin{equation*}
        F_3[s\vu+(1-s)\vv]\le sF_3[\vu]+(1-s)F_3[\vv].
    \end{equation*}

    Therefore, if $\frac{L}{a}<\beta$, for $s\in[0,1]$, the total energy $F[\vu]$ satisfies
    \begin{align}\label{eq..lambdaconvex}
        F[s\vu+(1-s)\vv]&=F_1[s\vu+(1-s)\vv]+F_2[s\vu+(1-s)\vv]+F_3[s\vu+(1-s)\vv]\nonumber\\
        &\le sF[\vu]+(1-s)F[\vv]-s(1-s)\frac{ac_1\beta-c_1L}{2}\|\nabla\nabla\cdot(\vu-\vv)\|_{L^2(\Omega)}^2\nonumber\\
        &\le sF[\vu]+(1-s)F[\vv]-s(1-s)\frac{ac_1\beta-c_1L}{2}\|\vu-\vv\|_{L^2(\Omega)}^2.
    \end{align}
    The second inequality comes from $\|\vu-\vv\|_{L^2(\Omega)}^2\le\|\nabla\nabla\cdot(\vu-\vv)\|_{L^2(\Omega)}^2$ due to the periodicity and zero average on $\Omega$ of $\nabla\cdot (\vu-\vv)$. Eq.~\eqref{eq..lambdaconvex} implies that $F[\vu]$ is $\lambda$-convex in $L^2(\Omega)$ with $\lambda=ac_1\beta-c_1L$.

    3. (Lower semi-continuity) Consider a sequence $\vu_n\rightarrow\vu$ weakly in $L^2(\Omega)$. We need to show
    \begin{equation*}
        \liminf\limits_{n\to +\infty}{F[\vu_n]}\geq F[\vu].
    \end{equation*}
    Assume that $\sup\limits_{n}F[\vu_n]<+\infty$, otherwise the inequality is trivial. The boundedness of energy $F[\vu_n]$ implies that $\nabla\nabla\cdot\vu_n$ is bounded in $L^2(\Omega)$. Therefore $\nabla\nabla\cdot\vu_n \rightarrow \nabla\nabla\cdot\vu$ weakly in $L^2(\Omega)$ and this lead to
    \begin{align*}
        &\liminf\limits_{n\to +\infty}-\frac{c_1}{2}\int_\Omega \nabla\cdot \vu_n(\vx)\int_{\sR^2}\frac{\vx-\vy}{\abs{\vx-\vy}^3}\cdot\nabla\nabla\cdot \vu_n(\vy)\diff{\vy}\diff{\vx}\\
        &=-\frac{c_1}{2}\int_\Omega \nabla\cdot \vu(\vx)\int_{\sR^2}\frac{\vx-\vy}{\abs{\vx-\vy}^3}\cdot\nabla\nabla\cdot \vu(\vy)\diff{\vy}\diff{\vx}
    \end{align*}
    by applying compact Sobolev embedding Theorem. The other term
    \begin{equation*}
        \int_\Omega\left\{\Psi(\nabla\nabla\cdot\vu+\vB)-a\frac{c_1\beta}{2}\abs{\nabla\nabla\cdot\vu+\vB}^2\right\}
        \diff{\vx}+a\frac{c_1\beta}{2}\|\nabla\nabla\cdot\vu+\vB\|_{L^2(\Omega)}^2
    \end{equation*}
    is convex and lower bounded. Thus, we conclude that $F[\vu]$ is lower semi-continuous with respect to the weak $L^2$-topology.

    4. (Compactness of sub-levels) This now follows directly from the lower semi-continuity of $F[\vu]$.
\end{proof}

\begin{proof}[Proof of \cref{thm..ExistenceEVI}]
    With \cref{prop..BoundednessConvexityLSCCompactness}, the theorem follows directly from \cite[Theorem 4.0.4]{Ambrosio2008Gradient}. (See Appendix~\ref{appendix:convexity} for this theorem in Ref.~\cite{Ambrosio2008Gradient}.)
\end{proof}

\begin{rmk}
     In Ref.~\cite{Gao2020regularity}, Gao et al. proved the solution existence of evolution variational inequality for $1+1$ dimensional continuum model using the gradient flow structure \cite{Ambrosio2008Gradient}. In that case, the $\lambda$-convexity of the total energy is naturally satisfied from the assumption of monotonically increasing height profile.
     Besides, their proof based on a modified PDE with all the coefficients to be of $O(1)$, unlike the multi-scale case we considered here.
\end{rmk}

\section{Energy scaling law}\label{sec5}
In this section, we consider the energy scaling law for the energy minimum state as the lattice constant $a\rightarrow 0$ compared with the length unit of the continuum model. This means that the number of steps $N\rightarrow \infty$ in a unit length of the continuum model.
The minimum energy scaling is obtained by finding proper upper and lower bounds. The lower bound is given by a series of inequalities, and the upper bound is established by a specific surface profile. We  also compare  the energy scaling law of the $2+1$ dimensional model with that of the $1+1$ dimensional model obtained in  Ref.~\cite{Luo2017Bunching}.

\subsection{Energy upper bound}\label{sec:upperbound}
We first consider the energy for a simple height profile whose slope along x-axis is constant and the undulation of steps along y-axis is periodic.
This special case
serves as an upper bound for the minimum energy scaling within the solution space.

Consider surface profile with the form
\begin{equation}\label{eq..SpecificProfile}
    h(x,y)=B\left(x+A\sin{\omega y}\right),
\end{equation}
where $A$ and $B$ are constants, $B$ is the average slope along $x$-axis, $AB\sin{\omega y}$ is the periodic deviation from the reference plane $Bx$, and $A,\omega>0$. We have
    $\nabla h=(B, AB\omega \cos{\omega y})$, $[\tilde{h}]_{H^{1/2}(\Omega)}^2=\frac{A^2B^2\omega L}{4\pi}$, where $\tilde{h}(x,y)=AB\sin{\omega y}$.
The total energy of this surface profile  on a periodic cell $\Omega=[0,L]\times[0,L]$ is
\begin{equation}\label{energyspecific}
    E[h]
    =-\frac{c_1\pi L^2}{2}A^2B^2\omega+a\int_\Omega\left(c_1\abs{\nabla h}\log(\abs{\nabla h}+\gamma_0)+c_2\abs{\nabla h}+c_3\abs{\nabla h}^3\right)\diff{\vx}.
\end{equation}

Assume that $A\gg 1$ and $\omega \sim O(1)$, as the lattice constant $a\rightarrow 0$ in the length unit of the continuum model. For fixed $B$, we choose $A$ such that the energy \eqref{energyspecific} is minimized. Such an $A$  satisfies $\frac{\partial E}{\partial A}=0$, which is
\begin{equation}\label{eq..criticalPointsA}
\begin{split}
  \frac{c_1\pi L}{\omega a}&=\int_0^L\left(\frac{c_1\log(\abs{\nabla h}+\gamma_0)}{\abs{\nabla h}}+\frac{c_1}{\abs{\nabla h}+\gamma_0}+\frac{c_2}{\abs{\nabla h}}+3c_3\abs{\nabla h}\right)\cos^2{\omega y}\diff{y}.
\end{split}
\end{equation}

We use dominant balance method to find the asymptotic behavior  as $a\to 0$ of the constant $A$ that satisfies Eq.~\eqref{eq..criticalPointsA}.
Under the assumptions $A\gg 1$ and $\omega \sim O(1)$, we have $A\omega\gg 1$  and $\abs{\nabla h}=O(A\omega)$ as $a\to 0$, thus
\begin{equation*}
    \frac{\log(\abs{\nabla h}+\gamma_0)}{\abs{\nabla h}},\quad \frac{1}{\abs{\nabla h}+\gamma_0},\quad \frac{1}{\abs{\nabla h}}\ll \abs{\nabla h} \quad \text{ a.e.}\quad \vx\in\Omega.
\end{equation*}
The dominant balance in \eqref{eq..criticalPointsA} is
\begin{flalign*}
   \frac{c_1\pi L^2}{a\omega}&\sim
   \int_0^L\int_0^L3c_3\abs{\nabla h}\cos^2{\omega y}\diff{x}\diff{y}\\
   &\sim L\int_0^L3c_3B\cos^2{\omega y}\cdot A\omega|\cos{\omega y}|\diff{y}\\
   &= \frac{2\omega L^2}{\pi}\int_0^{\frac{\pi}{2}}3c_3AB \cos^3{z}\diff{z}\\
   &=\frac{4c_3\omega L^2 B}{\pi}A.
\end{flalign*}
Thus,
$$A\sim\frac{c_1\pi^2}{4c_3\omega^2B}a^{-1}\gg 1, \ \ a\to 0,$$
which is consistent with the assumption. In this case, the minimum energy is
\begin{align*}
     E[h]
    &=-\frac{c_1\pi L^2}{2}A^2B^2\omega+ac_3 LB^3\int_0^L(1+A^2\omega^2\cos^2{\omega y})^{3/2}\diff{y}+O(\log{a})\\
    &\sim-\frac{c_1\pi L^2}{2}A^2B^2\omega+\frac{4ac_3L^2B^3\omega^3A^3}{3\pi}+O(\log a)\\
    &\sim-\frac{c_1^3\pi^5L^2}{96c_3^2\omega^3}a^{-2}+O(\log a).
\end{align*}

Therefore, we obtain an upper bound for the minimum energy in the solution space:
\begin{equation}\label{energy-upper}
    \inf\limits_{h\in X} E[h]\le -\frac{c_1^3\pi^5L^2}{96c_3^2\omega^3}a^{-2}+o\left(a^{-2}\right), \ \ a\to 0.
\end{equation}
\begin{rmk}
Note that the regularized logarithmic term does not appear in the dominant balance, thus the modified energy has the same energy scaling law as the original energy.
\end{rmk}

\subsection{Energy lower bound and proof of \cref{thm..scalinglaw}}
For the lower bound of the minimum energy, by using Eq.~\eqref{coer-E} in the proof of \cref{prop..Coercivity}, we have
\begin{flalign}
    E[h]
    &\ge -\frac{c_1L}{2}\|\nabla \tilde{h}\|_{L^2(\Omega)}^2+\int_\Omega ac_3\abs{\nabla \tilde{h}+\vB}^3 \diff{\vx}\nonumber\\
    &\ge \int_\Omega \left\{-\frac{c_1L}{2}\abs{\nabla \tilde{h}}^2+ac_3\left(\abs{\nabla\tilde{h}}-\abs{\vB}\right)^3\right\}\diff{\vx}\nonumber\\
    &\ge \int_\Omega \left\{-\frac{c_1L}{2}\abs{\nabla \tilde{h}}^2+ac_3\abs{\nabla \tilde{h}}^3-3ac_3\abs{\vB}\abs{\nabla \tilde{h}}^2\right\}\diff{\vx}-ac_3\abs{\vB}^3L^2\nonumber\\
    &\ge \min\limits_{\abs{\nabla \tilde{h}}}\left\{-\frac{c_1L}{2}\abs{\nabla \tilde{h}}^2+ac_3\abs{\nabla \tilde{h}}^3-3ac_3\abs{\vB}\abs{\nabla \tilde{h}}^2\right\}L^2-ac_3\abs{\vB}^3L^2\nonumber\\
    &=-\frac{c_1^3L^5}{54c_3^2}a^{-2}-\frac{c_1^2L^4\abs{\vB}}{3c_3}a^{-1}-2c_1L^2\abs{\vB}^2-5ac_3\abs{\vB}^3L^2
    \nonumber\\
    &=-\frac{c_1^3L^5}{54c_3^2}a^{-2}+o(a^{-2}), \ \ a\to 0. \label{energy-lower}
\end{flalign}
Here the minimum in the last inequality is obtained at $\abs{\nabla \tilde{h}(\vx_0)}=\frac{c_1L}{3c_3}a^{-1}+2\abs{\vB}$ for some $\vx_0\in \Omega$.

\begin{proof}[Proof of \cref{thm..scalinglaw}]
    Combining the upper and lower bounds  of the energy minimizer in Eq.~\eqref{energy-upper}  and \eqref{energy-lower}, the minimum energy scaling law in Eq.~\eqref{eqn:scalinglaw}  holds.
\end{proof}

\subsection{Physical meaning and competition of instabilities}\label{subsec:competition3}
From the energy scaling law \eqref{eqn:scalinglaw} in \cref{thm..scalinglaw} for the stepped surface in $2+1$ dimensions and the proofs shown above, it can be seen that the major contribution to the energy of the minimum energy surface profile is {\bf step meandering}, i.e., undulations along the steps, and the leading order energy is $O(a^{-2})$ as $a\to0$.
On the other hand, the $1+1$ dimensional model describes the $2+1$ dimensional case in which all the steps are straight. In this case, there is a different energy scaling law \eqref{eqn:scaling1d} which was obtained in Ref.~\cite{Luo2017Bunching}, and the major contribution to this energy is  {\bf step bunching} with leading order energy of $O(\log{a})$ as $a\to 0$. Our result shows that step meandering instability in general dominates over the step bunching instability in $2+1$ dimensions under  elastic effects. Below we give some quantitative comparisons between the energies due to these two instabilities.

As shown in \cref{sec:upperbound}, a surface profile with the form in \eqref{eq..SpecificProfile} is able to achieve an energy with the same order as the minimum energy surface in 2+1 dimensions described by the energy scaling law \eqref{eqn:scalinglaw}.
\cref{fig..3Dview+contour}(a) and (b) show an example of such surface profile and locations of steps
 (contour lines of the surface height).
 It can be seen that step undulation dominates on this surface.

\begin{figure}[htbp]
\centering
	\subfigure[]{\includegraphics[width=0.49\textwidth]{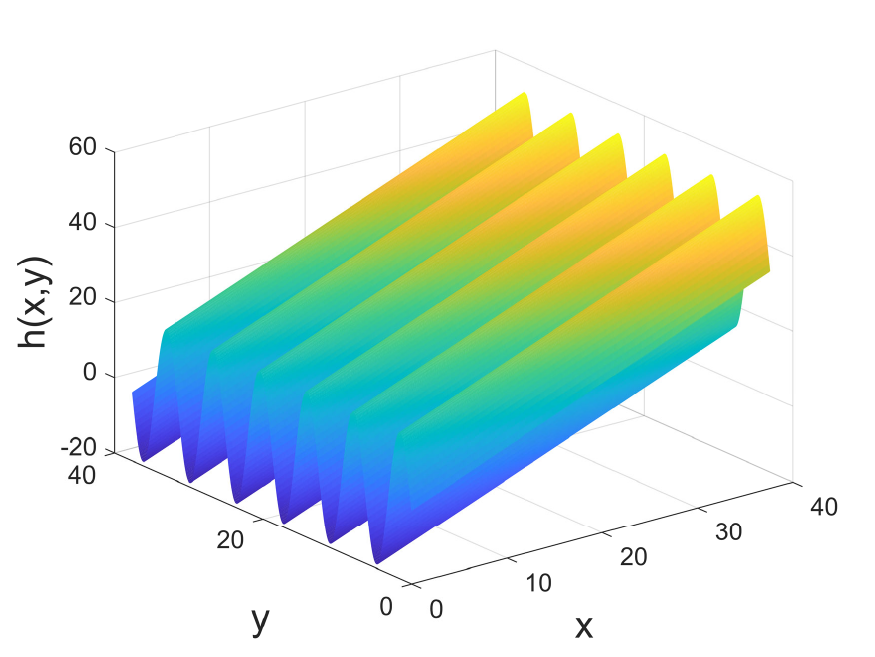}}
	\subfigure[]{\includegraphics[width=0.49\textwidth]{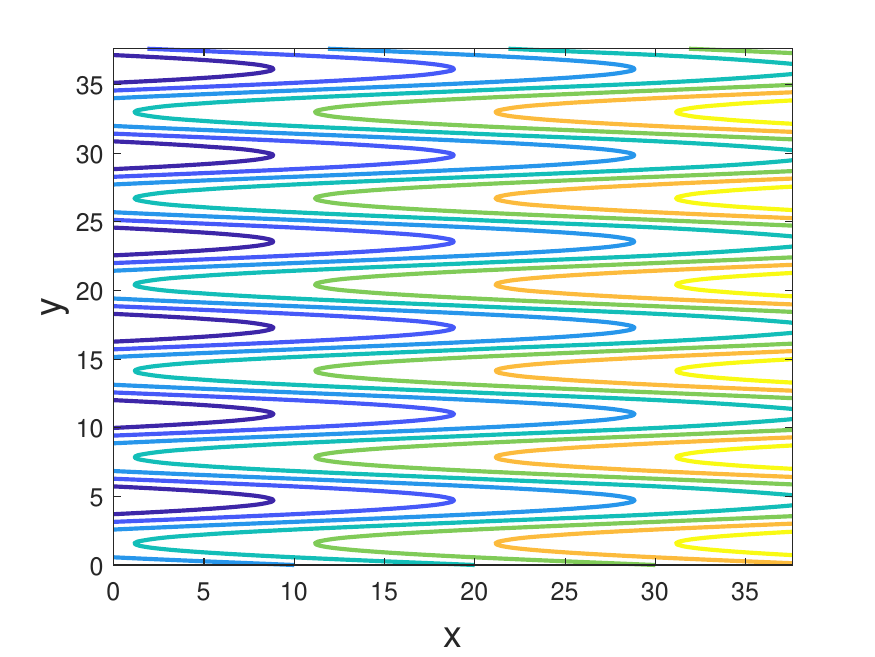}}
	\subfigure[]{\includegraphics[width=0.49\textwidth]{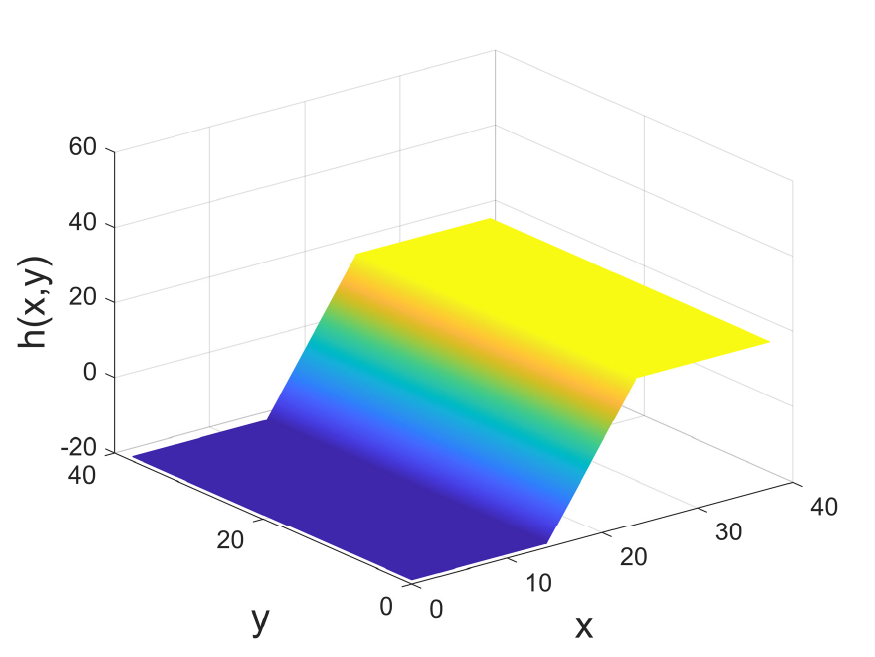}}
	\subfigure[]{\includegraphics[width=0.40\textwidth]{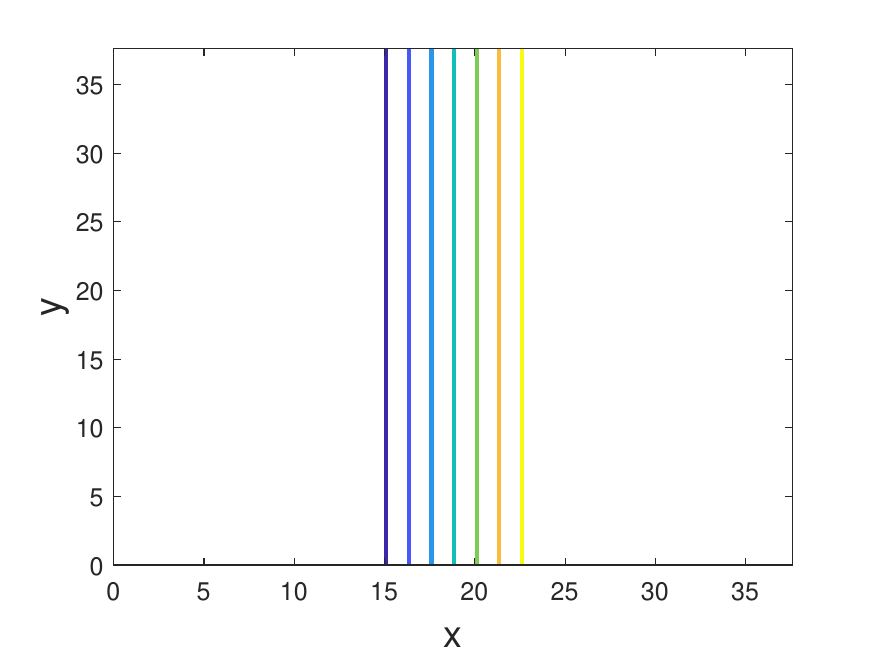}}
\caption{ (a) and (b): A surface profile with the step meandering dominant form in \eqref{eq..SpecificProfile} that achieves an energy with the same order as the minimum energy surface law \eqref{eqn:scalinglaw} in 2+1 dimensions, where $h(x,y)=x+6\pi\sin{y}$. (c) and (d): A step bunching surface profile with the one bunch form \eqref{eqn:one-bunch} obtained in \eqref{eq..SpecificProfile} that achieves an energy with the same order as the minimum energy surface law \eqref{eqn:scaling1d} in 1+1 dimensions, with $H=12\pi$ and $\rho=4$.
(a) and (c): Three-dimensional view of the surface. (b) and (d): Locations of steps.
The domain is $[0,12\pi]\times[0,12\pi]$. }
\label{fig..3Dview+contour}
\end{figure}

In Ref.~\cite{Luo2017Bunching}, it was shown that a surface profile with one bunch structure can achieve the minimum energy scaling law in 1+1 dimensions \eqref{eqn:scaling1d}.
The one bunch profile in a period $[0,L]$ is
\begin{equation}
 h(x)=\left\{
\begin{array}{ll}
-\frac{H}{2}, & {0\le x<\frac{L}2-\frac{H}{2\rho}},\vspace{1ex}\\
\rho \left(x-\frac{L}2\right), & {\abs{x-\frac{L}2}\le \frac{H}{2\rho}}, \vspace{1ex}\\
\frac{H}{2}, & \frac{L}2+\frac{H}{2\rho}<x\le L,
\end{array} \right. \label{eqn:one-bunch}
\end{equation}
where $H$ is the height of the step bunch, and $\rho>0$ is the step density within the step bunch.
\cref{fig..3Dview+contour}(c) and (d) show an example of such a step bunching surface profile and locations of steps.

Following the dominant balance analysis in \cref{sec:upperbound},
we take $A=\frac{c_1\pi^2}{4c_3\omega^2B}a^{-1}$
in the surface profile of the form \eqref{eq..SpecificProfile}. Under this condition, the 2+1 dimensional energy is
\begin{flalign}
    E_{2+1}=&-\frac{c_1\pi L^2}{2}A^2B^2\omega+a\int_\Omega\left[c_1\abs{\nabla h}\log(\abs{\nabla h}+\gamma_0)+c_2\abs{\nabla h}+c_3\abs{\nabla h}^3\right]\diff{\vx}\nonumber \\
    =&-\frac{c_1\pi L^2}{2}A^2B^2\omega\nonumber\\
    &+\frac{2aL^2}{\pi} \int_0^{\frac{\pi}{2}}\left[c_2B\sqrt{1+A^2\omega^2\cos^2{z}}+c_3B^3\left(\sqrt{1+A^2\omega^2\cos^2{z}}\right)^3\right.\nonumber\\
    &\left.+c_1B\sqrt{1+A^2\omega^2\cos^2{z}}\log\left(B\sqrt{1+A^2\omega^2\cos^2{z}}
    +\gamma_0\right)\right]\diff{z}, \label{E2+1asymp}
\end{flalign}
which achieves the energy scaling law of 2+1 dimensions with leading term of $O(a^{-2})$.

If all the steps are straight, $2+1$ dimensional continuum model is reduced to $1+1$ dimensional model (i.e., uniform in the direction of the steps). The energy in this case is
\begin{equation*}
    E_{1+1}=-c_1L\int_0^L h(x)\int_{\sR}\frac{h_x(y)}{x-y}\diff{y}\diff{x}+aL\int_0^L\left(c_1\abs{h_x}\log{\abs{h_x}}
    +c_2\abs{h_x}+c_3\abs{h_x}^3\right)\diff{x}.
\end{equation*}
According to the calculations in Ref.~\cite{Luo2017Bunching} for finding the 1+1 dimensional energy scaling law, we take step density $\rho=\sqrt{\frac{c_1H}{2c_3}}a^{-\frac{1}{2}}$ in the surface profile of the form \eqref{eqn:one-bunch}. Under this condition, the energy $E_{1+1}$ is
\begin{flalign}
E_{1+1}
=&c_1L\rho^2\int_{-\frac{H}{2\rho}}^{\frac{H}{2\rho}}\int_{-\frac{H}{2\rho}}^{\frac{H}{2\rho}}
\log{\sin(\frac{\pi(x-y)}{L})}\diff{y}\diff{x}\nonumber \\
&+aL\int_{-\frac{H}{2\rho}}^{\frac{H}{2\rho}}
\left(c_1\rho\log{\rho}+c_2\rho+c_3\rho^3\right)\diff{x}\nonumber\\
=&c_1LH^2\log\left(\frac{\pi H}{L\rho}\right)+aLH(c_1\log{\rho}+c_2+c_3\rho^2), \label{E1+1asymp}
\end{flalign}
which achieve the energy scaling law of 1+1 dimensions with leading term of $O(\log{a})$.

Now we compare the energies $E_{2+1}$ in \eqref{E2+1asymp} and $E_{1+1}$ in \eqref{E1+1asymp}, which give the correct asymptotic behaviors of the minimum energy in $2+1$ dimensions and the minimum energy for surfaces with straight steps, respectively.
In the comparisons, the domain length $L=Nl_t$ with $N$ being the number of steps in the domain and $l_t$ being the average distance between adjacent steps. Accordingly, the height increase over the domain  $H=Na$, the average slope  $B=\frac{a}{l_t}$, and $\omega=\frac{2\pi}{L}$. The parameters $c_1=\frac{(1-\nu)\sigma_0^2}{2\pi G}$, $c_2=\frac{g_1}{a}+c_1\log\frac{2\pi r_c}{\E a}$, and $c_3 a=\frac{g_3}{3}$ with $\sigma_0=\frac{2G(1+\nu)\eps_0}{1-\nu}$. For the values of the parameters, $g_1=\SI{0.03}{\joule\per \meter^2}$, $g_3=\SI{8.58}{\joule\per \meter^2}$, the lattice height $a=\SI{0.27}{\nano\meter}$, the elastic moduli $\nu=0.25$ and $G=\SI{3.8e10}{\pascal}$. The core parameter $r_c$ of a step is assumed to be $a$.

\begin{figure}[htbp]
\centering
	\subfigure[]{\includegraphics[width=0.49\textwidth]{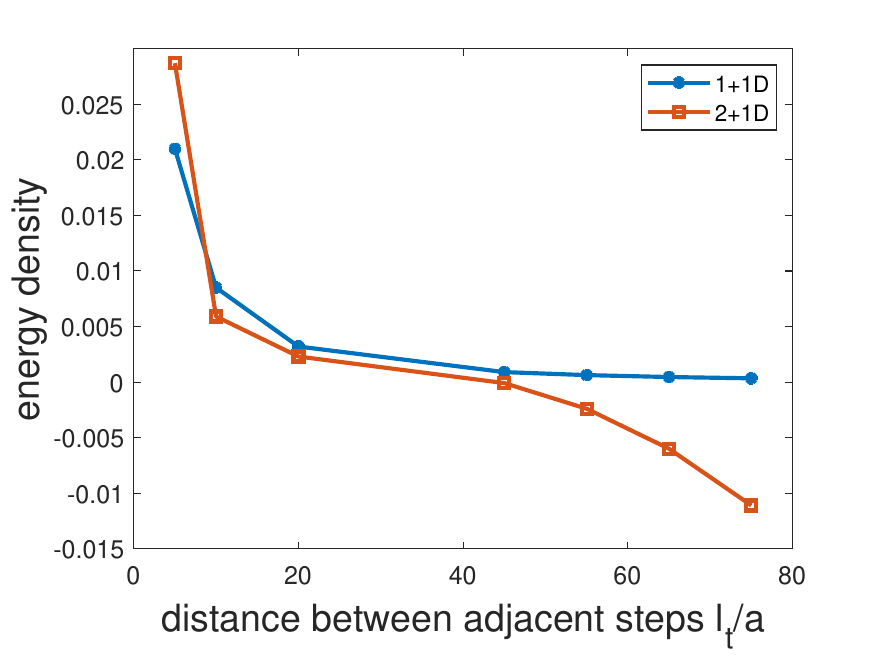}}
	\subfigure[]{\includegraphics[width=0.49\textwidth]{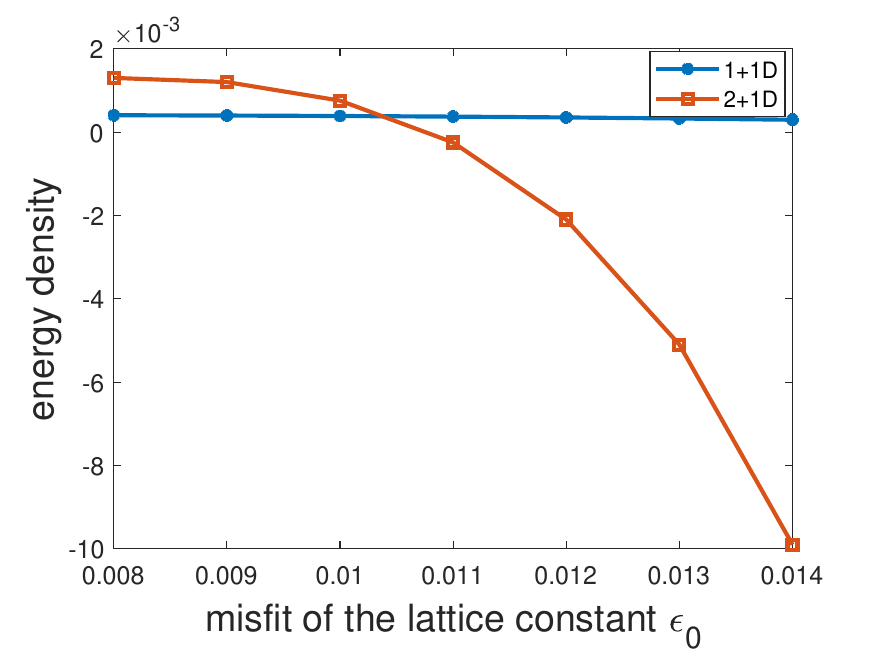}}
\caption{Minimum energy density $(J/{m^2})$ comparisons in terms of the distance between adjacent steps $l_t $ (for fixed $N=15$, $\epsilon_0=0.012$) in (a) and the misfit $\epsilon_0$ (for fixed $N=10$, $l_t=80a$) in (b), for step undulation dominated surface (2+1 D) with energy $E_{2+1}$ in \eqref{E2+1asymp} and step bunching dominated surface (1+1 D) with energy $E_{1+1}$ in \eqref{E1+1asymp}.}
\label{fig..energy_comparison}
\end{figure}

The minimum energy comparisons are summarized in \cref{fig..energy_comparison}, with different values of the adjacent distance $l_t$ and the misfit $\epsilon_0$. \cref{fig..energy_comparison}(a)  shows that for fixed misfit $\epsilon_0=0.012$, when $l_t$ is small, minimum energy of the surface with step bunching $E_{1+1}$ is smaller than the minimum energy of surface with step undulation $E_{2+1}$, which means that the step bunching instability dominates; when $l_t$ is large, $E_{2+1}<E_{1+1}$, which means that the step meandering instability dominates. As the $l_t$ increases, there exists a transition from step bunching instability to step meandering instability. \cref{fig..energy_comparison}(b) shows that for fixed adjacent distance $l_t=80a$, as the misfit $\epsilon_0$ increases, there also exists a transition from step bunching to step undulation.
These results show that step meandering dominates over step bunching in $2+1$ dimensions in general except for small inter-step distance $l_t$ and small misfit $\epsilon_0$.
These competitions between the two different step instabilities  in terms of energy are consistent with the results of linear instability analysis and numerical simulations obtained in Ref.~\cite{Zhu2009Continuum}. Note that the conclusions on competition of instabilities obtained here are based on the global energy minimizer, whereas those presented in Ref.~\cite{Zhu2009Continuum} are based on evolution tendency of the surface.

\section{Conclusion}\label{sec6}
In this paper, we have studied the continuum model for  epitaxial surfaces in 2+1 dimensions under elastic effects obtained in Ref.~\cite{Xu2009Derivation}. We have proposed a modified continuum model to fix the inaccurate formulation of the  energy in the small $|\nabla h|$ regime based on the underlying physics. This modification solves the problem of
possible illposedness due to the nonconvexity (in terms of the gradient of the surface) of the energy functional. The illposedness associated with nonconvexity in the original continuum model is in general not in the physical regime and our modification only leads to negligible change under the physically meaningful setting.

For the modified continuum model, we have proved the existence and uniqueness of the energy minimizer  by  coercivity and lower semi-continuity in the framework of calculus of variations. The existence and uniqueness of weak solution of the corresponding evolution equation has also been established based on the framework of gradient flow.

We have also obtained the minimum energy scaling law for the 2+1 dimensional epitaxial surfaces under elastic effects, which is attained by surfaces with step meandering instability and is essentially different from the energy scaling law for the 1+1 dimensional epitaxial surfaces under elastic effects~\cite{Luo2017Bunching} attained with step bunching surface profiles. Transition from the step bunching instability (where all steps are straight) to the step meandering instability has been discussed. Since the $2+1$ continuum model was derived from the corresponding discrete model as an asymptotic approximation~\cite{Xu2009Derivation}, it is expected that these minimum energy scaling laws and transition between the two surface instabilities also hold for the corresponding discrete model~\cite{Xu2009Derivation} (cf. \cite{Houchmandzadeh1995,Leonard2003,Tersoff1992}), with some adjustments in the proofs and calculations in \cref{sec5}.

\appendix

\section{Comparisons of continuum models with and without the modification}\label{appendix:regularization}

In this section of appendix, we further validate the modification by comparisons in linear instability analysis and numerical simulation by using continuum models with and without the modification.

Let the domain size be $L = Nl_t$, where $N$ is the number of steps in the domain and $l_t$ is the average distance between adjacent steps. Scaling $x$ and $y$ by $\frac{Nl_t}{2\pi}$, $h$ by $\frac{Na}{2\pi}$, and $t$ by $\frac{(Nl_t)^3Na}{g_1(2\pi)^4}$,
from Eqs.~\eqref{eqn:evolution-m} and \eqref{eqn:mu-m1},
we have the following dimensionless evolution equation after modification:
\begin{flalign}
    \frac{\partial h}{\partial t} &= \Delta\left[-\nabla\cdot\left(\frac{\nabla h}{\abs{\nabla h}}+\alpha_1\abs{\nabla h}\nabla h\right)\right.\nonumber\\
    &-\alpha_2\int_{-\infty}^{\infty}\int_{-\infty}^{\infty}\frac{(x-\xi)h_x(\xi,\eta)+(y-\eta)h_y(\xi,\eta)}{[(x-\xi)^2+(y-\eta)^2]^{3/2}}\diff{\xi} \diff{\eta}\nonumber\\
    &\left.-\alpha_3\nabla\cdot\left(\frac{\nabla h}{\abs{\nabla h}}\log\left(\frac{\alpha_4}{e}\big(\abs{\nabla h}+\gamma_0l_t/a\big)\right)+\frac{\nabla h}{\abs{\nabla h}+\gamma_0l_t/a}\right)\right], \label{eqn:dimensionless}
\end{flalign}
where the dimensionless constants
\begin{flalign}
    \alpha_1 = \frac{g_3a^2}{g_1l_t^2},\quad \alpha_2=\frac{(1-\nu)\sigma_0^2Na}{4\pi^2g_1G}, \quad \alpha_3 = \frac{(1-\nu)\sigma_0^2a}{2\pi g_1G},\quad \alpha_4 = \frac{2\pi r_c}{l_t}.
\end{flalign}

The linear instability analysis examines the short-term behavior of  a planar surface (representing a uniform straight step array) subject to small perturbations. Consider such a surface profile given by
\begin{equation}
    h(x,y,t) = -x +\epsilon e^{ik_1x+ik_2y+\omega t},
\end{equation}
where the amplitude of the perturbation $\epsilon$ is very small. Inserting this expression into
the dimensionless equation \eqref{eqn:dimensionless} and keeping the $O(\epsilon)$ terms, we obtain the following
dispersion relation:
\begin{equation}
    \omega = \omega_{origin}+\omega_{diff},
\end{equation}
where $\omega_{origin}$ is the dispersion relation without modification and  $\omega_{diff}$ is the difference:
\begin{flalign}
    \omega_{origin} &= (k_1^2+k_2^2)[2\pi \alpha_2(k_1^2+k_2^2)^{1/2} -(2\alpha_1+\alpha_3)k_1^2-(1+\alpha_1+\alpha_3\log{\alpha_4})k_2^2],\\
    \omega_{diff} &= \alpha_3(k_1^2+k_2^2)\left[\frac{(\gamma_0 l_t/a)^2}{(1+\gamma_0 l_t/a)^2} k_1^2-\left(\log(1+\gamma_0 l_t/a)-\frac{\gamma_0 l_t/a}{1+\gamma_0 l_t/a}\right)k_2^2\right].
\end{flalign}
Notice that when $\gamma_0\ll a/l_t$, i.e., $\gamma_0 l_t/a\ll 1$,  the difference $\omega_{diff}$ is of order $O((\gamma_0 l_t/a)^2)$. That is, the modification leads to very small difference in the linear instability relation.

\begin{figure}[htbp]
\centering
	\subfigure[]{\includegraphics[width=0.30\textwidth]{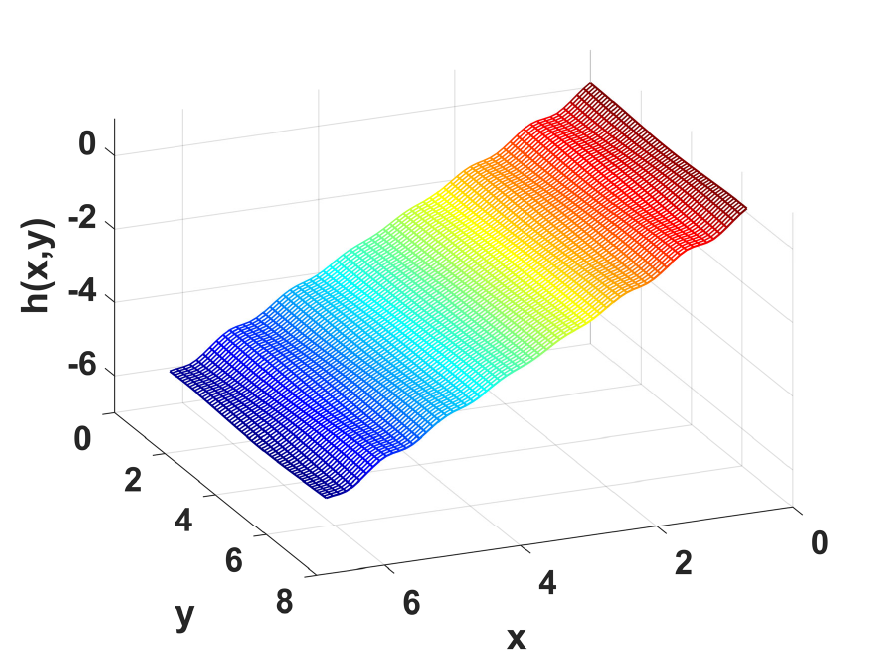}}
	\subfigure[]{\includegraphics[width=0.30\textwidth]{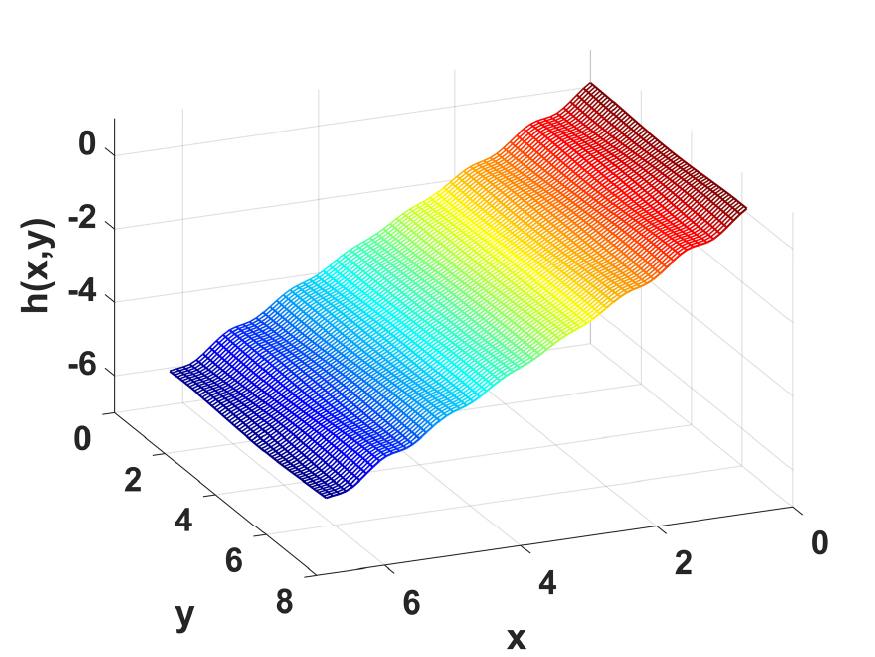}}
	\subfigure[]{\includegraphics[width=0.30\textwidth]{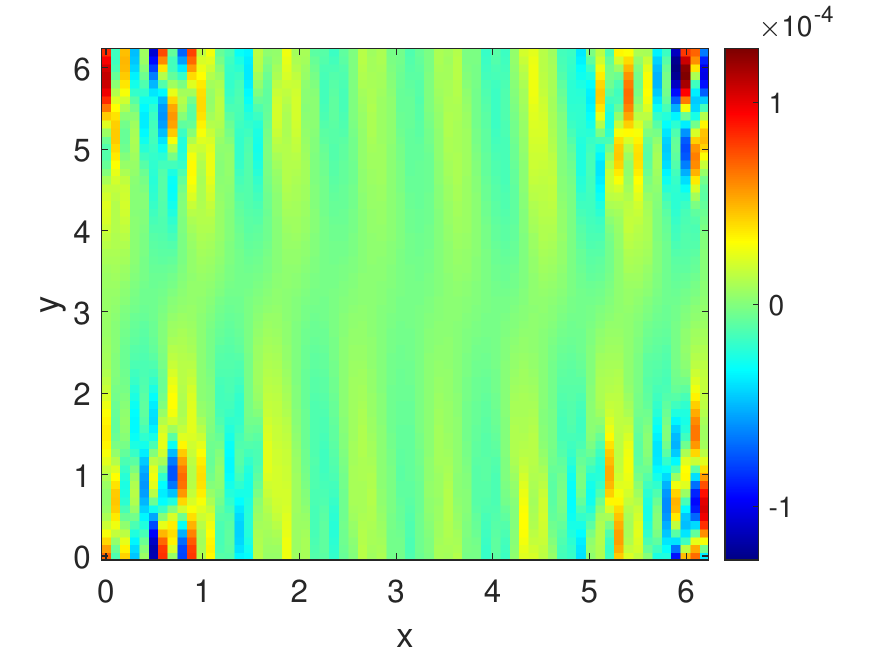}}
	\subfigure[]{\includegraphics[width=0.30\textwidth]{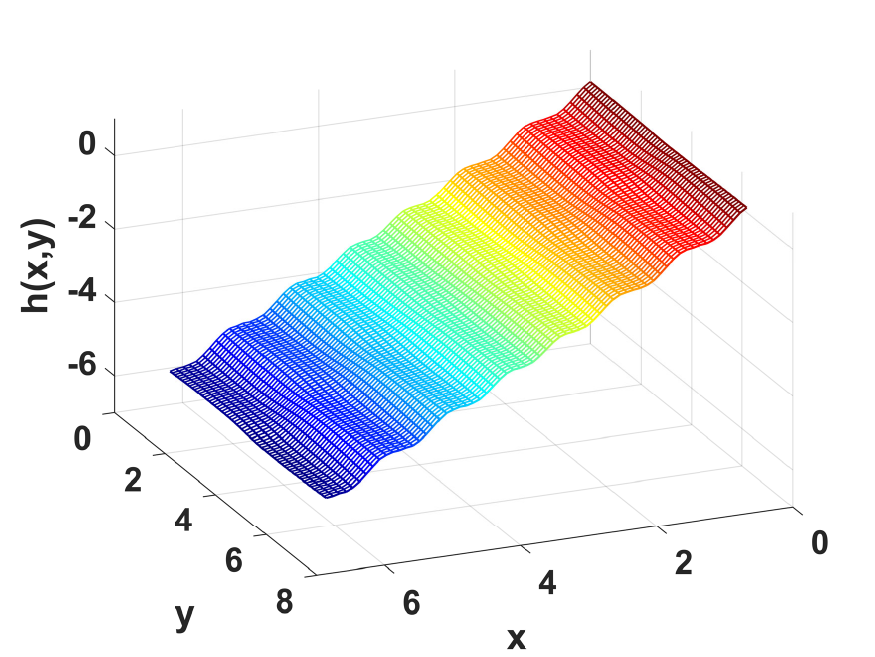}}
	\subfigure[]{\includegraphics[width=0.30\textwidth]{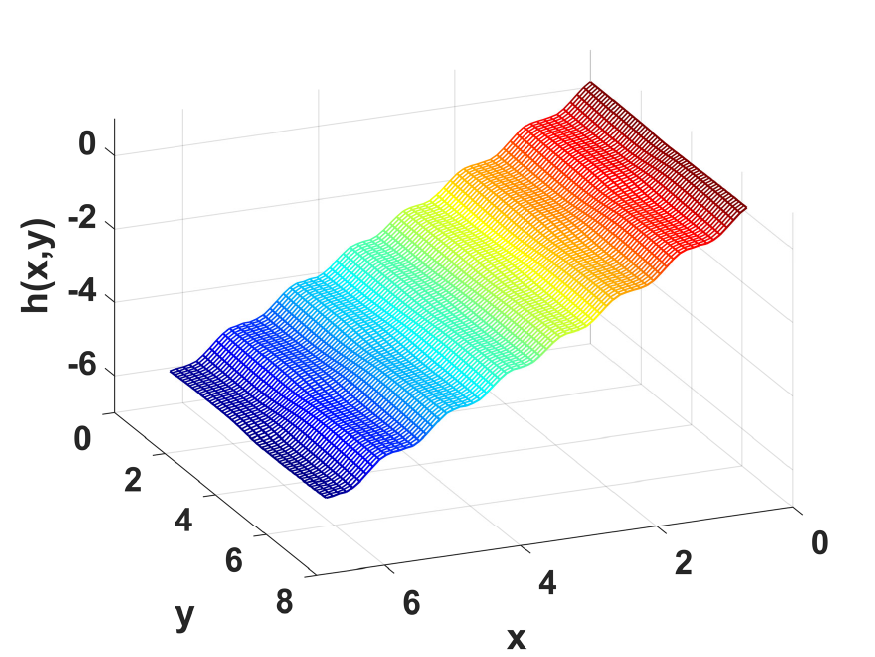}}
	\subfigure[]{\includegraphics[width=0.30\textwidth]{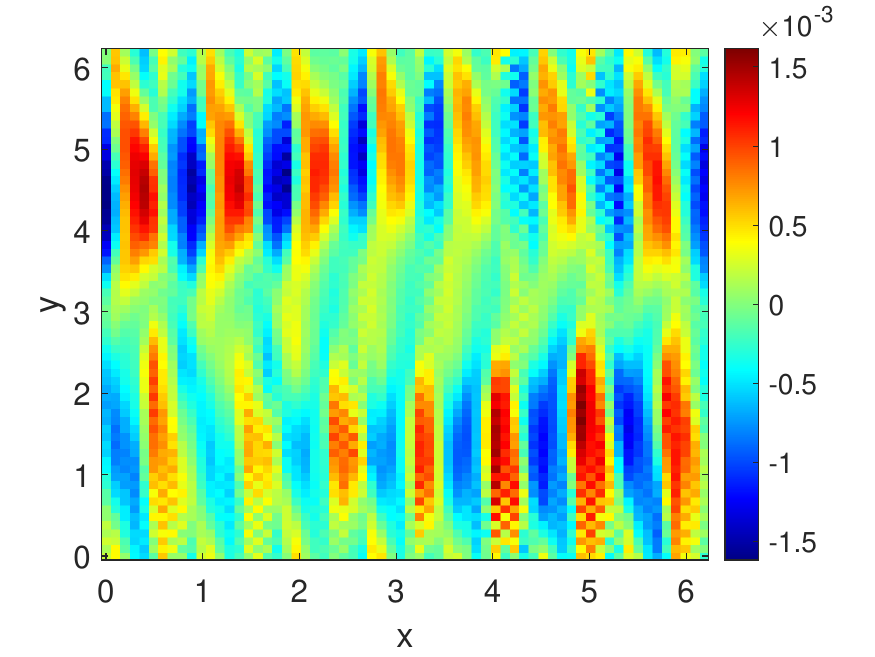}}
	\subfigure[]{\includegraphics[width=0.30\textwidth]{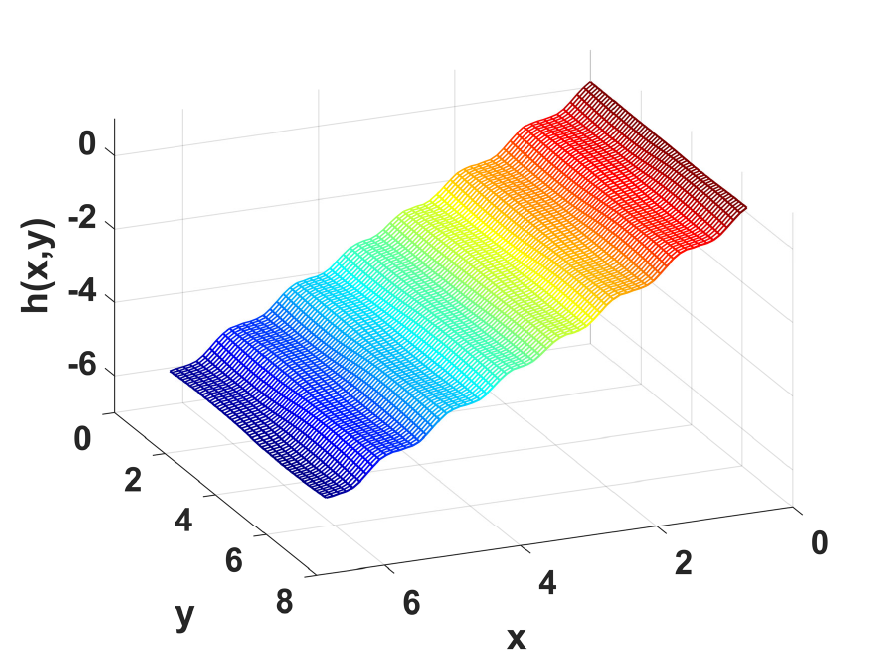}}
	\subfigure[]{\includegraphics[width=0.30\textwidth]{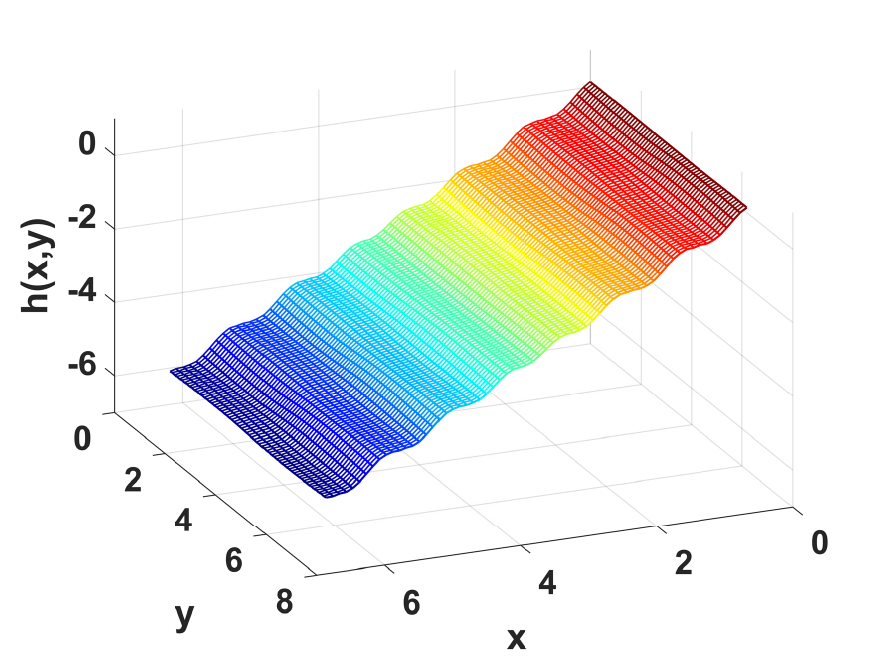}}
	\subfigure[]{\includegraphics[width=0.30\textwidth]{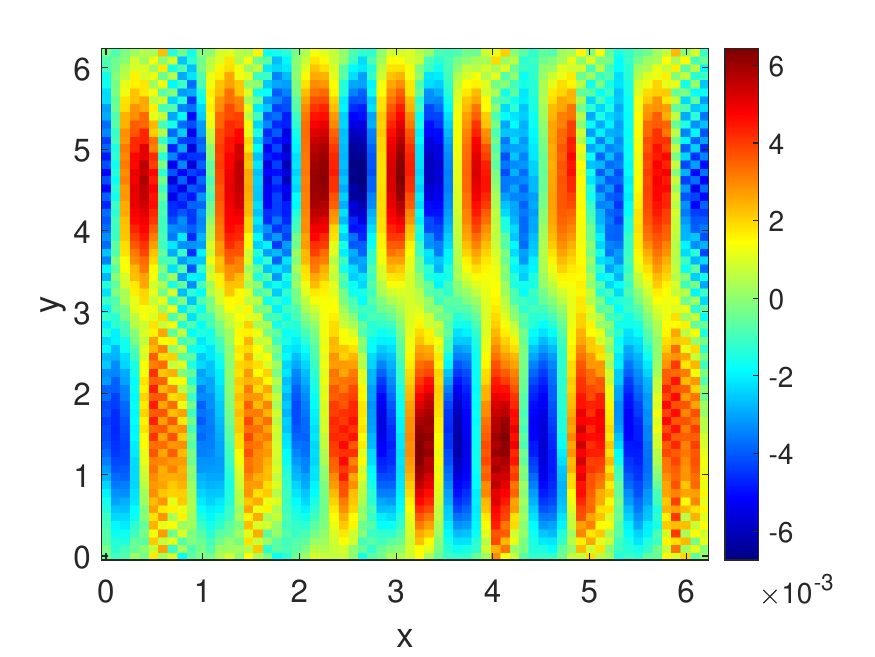}}
	\subfigure[]{\includegraphics[width=0.30\textwidth]{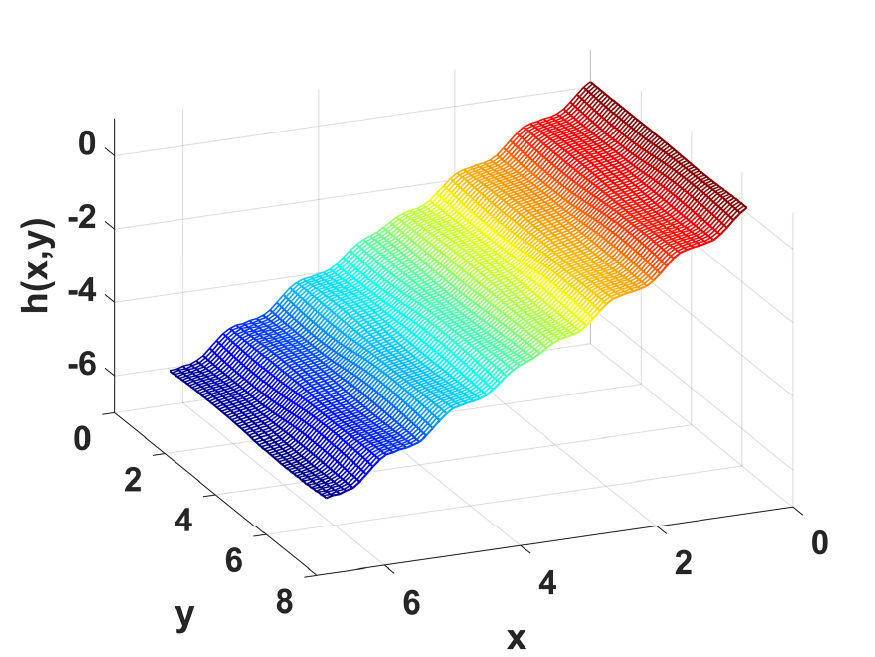}}
	\subfigure[]{\includegraphics[width=0.30\textwidth]{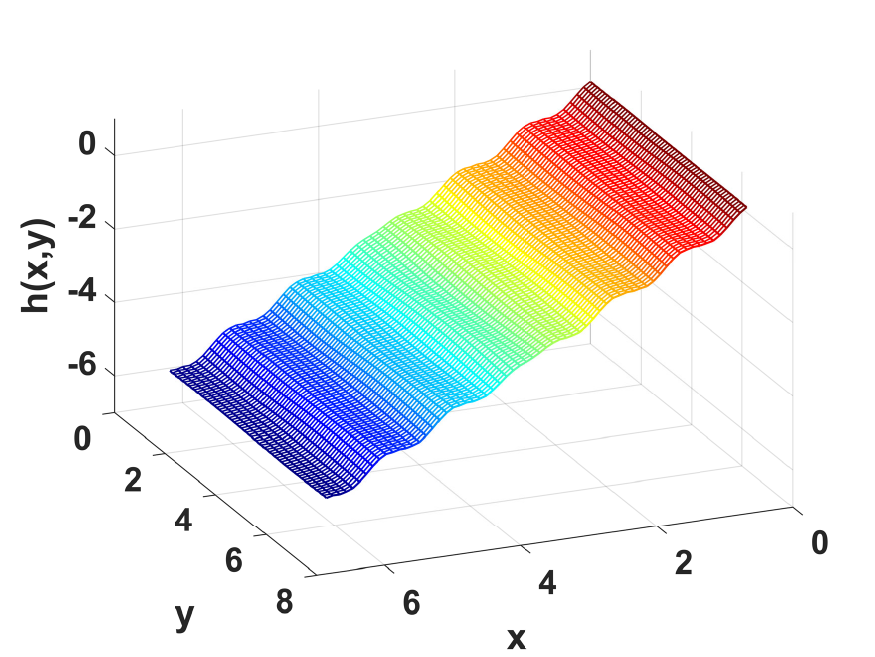}}
	\subfigure[]{\includegraphics[width=0.30\textwidth]{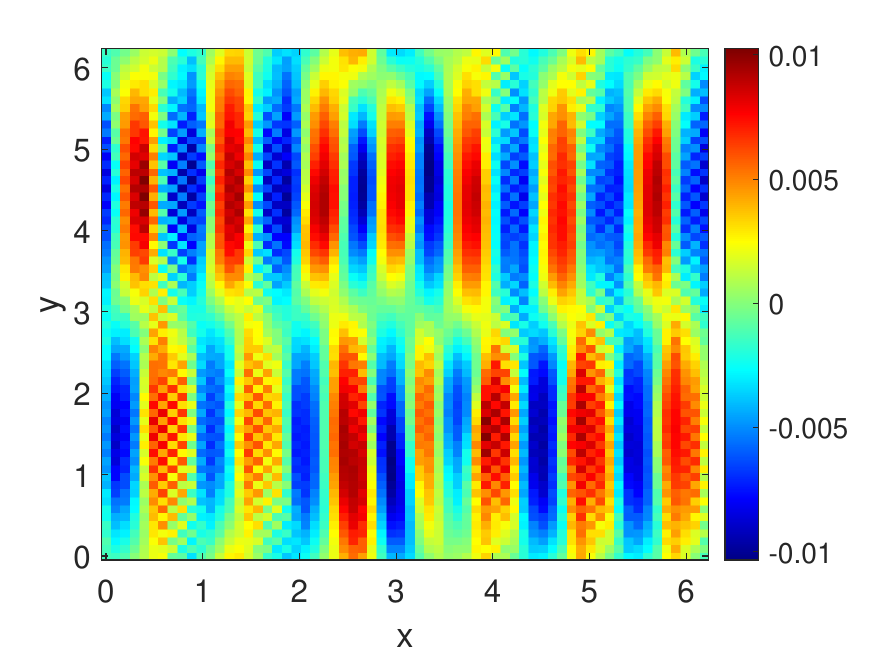}}
\caption{Time evolution of the surface with parameters $\epsilon_0=0.012$, $N=40$, $l_t=50a$, and  initial profile in Eq.~\eqref{eqn:ini}. (a)–(c): t=0; (d)–(f) t=0.0225; (g)–(i): t=0.0315; (j)–(l): t=0.0395; (m)–(o) t=0.0504. The left column of images show the three-dimensional view before modification, the column
of images in the middle show the three-dimensional view after modification, and the right column of images
show the difference of the surface height introduced by modification (with respect to the maximum height difference).}
\label{fig..3Dview-bunchingD}
\end{figure}

We also perform numerical simulations for the surface evolution beyond the linear instability regime by using the evolution equations with and without modification, and the comparison results are shown in Fig \ref{fig..3Dview-bunchingD}.
The simulation domain is $[0,2\pi]\times[0,2\pi]$, divided uniformly into $64\times64$ grid points. Here
$\epsilon_0=0.012$, $N=40$, $l_t=50a$, and the modification parameter $\gamma_0=9.7\times10^{-8}$. Accordingly, the small parameter in the dimensionless equation $\gamma_0 l_t/a=4.85\times10^{-6}$. Values of other parameters are the same as the example in Sec.~\ref{subsec:competition3}.
Initially, the surface is a planar one subject to some small perturbations with wave numbers near a wave-number pair $(k_1^0, k_2^0)=(7,1)$ in the unstable region from linear instability analysis:
\begin{equation}\label{eqn:ini}
    h(x,y) = -x+\sum\limits_{\abs{k_1-k_1^0}\le 1,\abs{k_2-k_2^0}\le 1}\frac{10^{-3}}{[(k_1-k_1^0)^2+(k_2-k_2^0)^2]^{3/2}+1}\cos(k_1 x+k_2 y).
\end{equation}
where $(k_1, k_2)$ is a wave-number pair in the unstable region in
linear instability analysis, and $k_1^0,k_2^0= 1$ or $2$.
It can be seen that the relative different between the results using the two equations remains small during the evolution.

\section{Definition of $\lambda-$convexity and Theorem 4.0.4 in Ref.~\cite{Ambrosio2008Gradient}}\label{appendix:convexity}
\begin{defi}[$\lambda-$convexity \cite{Ambrosio2008Gradient}] In a metric space $(\mathscr{S},d)$,
a functional $\phi: \mathscr{S}\to (-\infty,+\infty]$ is called $\lambda-$convex on a curve $\gamma: t\in[0,1]\mapsto \gamma_t\in \mathscr{S}$ for some $\lambda\in\sR$ if
\begin{equation*}
    \phi(\gamma_t)\le(1-t)\phi(\gamma_0)+t\phi(\gamma_1)-\frac{1}{2}\lambda t(1-t)d^2(\gamma_0,\gamma_1),\quad \forall t\in[0,1]
\end{equation*}
\end{defi}

\begin{thmm}[Generation and main properties of the evolution semigroup~\cite{Ambrosio2008Gradient}]\label{weak solution existence}\leavevmode \\
Assume $(\mathscr{S},d)$ is a complete metric space and $\phi :\mathscr{S}\to (-\infty,+\infty]$ is a proper, coercive, lower semicontinuity functional. Furthermore, for every choice of $\omega, v_0, v_1\in D(\phi)$, there exists a curve $\gamma = \gamma_t, t\in[0,1]$ with $\gamma_0 = v_0, \gamma_1 = v_1$ such that for some $\lambda\in\sR$,
\begin{equation*}
    v\mapsto \Phi(\tau,\omega;v)=\frac{1}{2\tau}d^2(v,\omega)+\phi(v)
\end{equation*}
is $(\tau^{-1}+\lambda)-$convex on $\gamma$ for each $\tau$ such that $\tau^{-1}+\lambda>0$.
Then we have
\begin{itemize}
    \item[i.] Uniqueness and evolution variational inequalities: $u$ is the unique solution of the evolution variational inequality
    \begin{equation*}
        \frac{1}{2}\frac{\diff}{\diff{t}}d^2(u(t),v)+\frac{1}{2}\lambda d^2(u(t),v)+\phi(u(t))\le \phi(v) \quad\mathscr{L}^1-a.e. ~t>0,\forall v\in D(\phi).
    \end{equation*}
    among all the locally absolutely continuous curves such that $\lim\limits_{t\downarrow 0}u(t) = u_0$ in $\mathscr{S}$. Where $D(\phi):=\{v\in\mathscr{S}:\phi(v)<+\infty\}\neq \emptyset$
    \item[ii.] Regularizing effect: $u$ is a locally Lipschitz curve of maximal slope with $u(t)\in D(\abs{\partial \phi})\subset D(\phi)$ for $t>0$; in particular, if $\lambda \ge 0$, the following a priori bounds hold:
    \begin{align*}
        &\phi(u(t))\le \phi_t(u_0)\le \phi(v)+\frac{1}{2t}d^2(v,u_0) \quad \forall v\in D(\phi),\\
        &\abs{\partial \phi}^2(u(t))\le \abs{\partial \phi}^2(v)+\frac{1}{t^2}d^2(v,u_0) \quad \forall v\in D(\abs{\partial \phi}).
    \end{align*}
\end{itemize}
\end{thmm}

\section*{Acknowledgments}
This work was supported by the Hong Kong Research Grants Council General Research Fund 16313316 and the Project of Hetao Shenzhen-HKUST Innovation Cooperation Zone
HZQB-KCZYB-2020083 (Y.X.), the National Natural Science Foundation of China Grant No. 12101401  and Shanghai Municipal of Science and Technology Major Project No. 2021SHZDZX0102 (T.L.).


\end{document}